\documentclass[a4paper,11pt
 ,parskip=half
 ,appendixprefix=true
 ,abstract=true
,%
]{amsart}
\usepackage{libertine}
\usepackage{amssymb,latexsym} 
\usepackage[T1]{fontenc} 
\usepackage[latin9]{inputenc}
\usepackage{amsmath} 
\usepackage{amsfonts} 
\usepackage{amsthm}
\usepackage{enumerate} 
\usepackage{ifthen} 
\usepackage{amssymb}
\usepackage{comment}
\usepackage{bibentry}
\usepackage{tikz}
\usetikzlibrary{matrix,arrows}

\newcounter{mypcounter}
\numberwithin{equation}{section}

\begin{document}

\theoremstyle{plain} 
\newtheorem{theorem}{Theorem}[section]
\newtheorem{proposition}[theorem]{Proposition}
\newtheorem{lemma}[theorem]{Lemma}
\newtheorem{corollary}[theorem]{Corollary}

\renewcommand{\themypcounter}{\Roman{mypcounter}}
\newtheorem{problem}[mypcounter]{Problem}
\theoremstyle{definition} 
\newtheorem{definition}[theorem]{Definition}

\theoremstyle{remark} 
\newtheorem{remark}[theorem]{Remark}
\newtheorem*{example}{Example}

\renewcommand{\thesection}{\arabic{section}}

\newcommand{\mi}{\ensuremath{\mathrm{i}}}
\newcommand{\me}{\ensuremath{\mathrm{e}}}
\newcommand{\mPS}{\ensuremath{(\Omega,\mathcal{F},P)}}
\newcommand{\abs}[1]{\left| #1 \right|}
\newcommand{\id}{\operatorname{I}}

\newcommand{\Ito}{It\^o}
\newcommand{\cadlag}{c\`adl\`ag}
\newcommand{\Levy}{L\'evy}
\newcommand{\Kac}{Kac}
\newcommand{\Oksendal}{\O{}ksendal}
\newcommand{\Lebesgue}{Lebesgue}
\newcommand{\Fatou}{Fatou}
\newcommand{\Erdelyi}{Erd\'elyi}
\newcommand{\norm}[1]{\lVert #1 \rVert}
\newcommand{\diff}[1]{\operatorname{d}\ifthenelse{\equal{#1}{}}{\,}{\!#1}}
\newcommand{\E}{\operatorname{E}}
\newcommand{\charop}{\operatorname{\mathbb{L}}} 
\newcommand{\KummerM}{\operatorname{M}}
\newcommand{\KummerU}{\operatorname{U}}
\newcommand{\WhittakerM}{\operatorname{M}}
\newcommand{\WhittakerW}{\operatorname{W}}
\newcommand{\WeberD}{\operatorname{D}}
\newcommand{\HermiteH}{\operatorname{H}}
\newcommand{\cP}[2]{\ensuremath{P \left( #1 \mid #2 \right)}}
\newcommand{\cE}[2]{\ensuremath{\operatorname{E} \left( #1 \mid #2
    \right)}}
\newcommand{\cov}{\operatorname{Cov}}
\newcommand{\ind}{1{\hskip -2.5 pt}\hbox{\textnormal{I}}}
\newcommand{\on}[1]{\operatorname{#1} }

\newcommand{\C}{\mathbb{C}} 
\newcommand{\R}{\mathbb{R}}
\newcommand{\Q}{\mathbb{Q}} 
\newcommand{\N}{\mathbb{N}}

\title[Optimal Convergence Rates and One-Term Edgeworth Expansions]{Optimal Convergence Rates and One-Term Edgeworth Expansions for
  Multidimensional Functionals of Gaussian Fields}
\author{Simon Campese}
\date{\today}
\address{Universit\'e du Luxembourg\\ Facult\'e des Sciences, de la Technologie
  et de la Communication\\ Unit\'e de Recherche en Math\'emathiques\\ \\ 6, rue
  Coudenhove Kalergi \\ 1359 Luxembourg}
\email{simon.campese@uni.lu}
 \thanks{Part of this research was supported by the Fonds National de
    la Recherche Luxembourg under the accompanying measure AM2b}
\maketitle
\nobibliography*

\begin{abstract}
  We develop techniques for determining the exact asymptotic speed
  of convergence in the multidimensional normal approximation of
  smooth functions of Gaussian fields. As a by-product, our findings
  yield exact limits and often give rise to one-term generalized
  Edgeworth expansions increasing   the speed of convergence. Our main
  mathematical tools are Malliavin calculus, 
  Stein's method and the Fourth Moment Theorem. 
  This work can be seen as an extension of the results
  of~\cite{1196.60034} to the multi-dimensional case, with the notable 
  difference that in our framework covariances are allowed to
  fluctuate. We apply our findings to exploding functionals of
  Brownian sheets, vectors of Toeplitz quadratic functionals and the
  Breuer-Major Theorem.
\end{abstract}

\section{Introduction}
\label{s-2}

Let $X$ be an isonormal Gaussian process on some real, separable
Hilbert space $\mathcal{H}$ and $(F_n)$ be a sequence of centered,
real-valued functionals of $X$ with converging 
covariances. Moreover, assume that $F_n \xrightarrow{\mathcal{L}} Z$,
where $Z$ is a centered Gaussian random variable and
$\xrightarrow{\mathcal{L}}$ denotes convergence in
law. In~\cite{MR2520122}, Nourdin and Peccati
used a combination of Stein's  method (see~\cite{pecnour_mall},  \cite{MR2732624}, \cite{MR2235448},
\cite{MR2235451}, \cite{MR882007}, \cite{MR0402873}) and Malliavin
calculus (see~\cite{pecnour_mall}, \cite{MR2200233}, \cite{MR1474726})
to derive the bound
\begin{equation}
  \label{eq:15}
  \abs{
    \E \left[ g(F_n) \right]
    -
    \E \left[ g(Z) \right]
  }
  \leq
  M(g)
  \,
  \varphi(F_n)
\end{equation}
and used it to prove estimates for several probabilistic distances
$d(F_n,Z)$ (among them the Fortet-Mourier, Kolmogorov and Wasserstein
distances). The quantity $\varphi(F_n)$ in the bound~\eqref{eq:15} is defined by 
\begin{equation*}
  \varphi(F_n) =
    \sqrt{
      \on{Var}
      \left\langle DF_n,-DL^{-1}F_n
      \right\rangle_{\mathfrak{H}} 
    }
    +
    \abs{
      \E \left[ F_n^2 \right] - \E \left[ Z^2 \right]
      },
\end{equation*}
where $g$ is a  sufficiently smooth function, $M(g)$ is a
constant depending on $g$ and the random variable
$\left\langle DF_n,-DL^{-1}F_n \right\rangle_{\mathfrak{H}}$ involves the
Malliavin derivative operator $D$  and the pseudo-inverse $L^{-1}$ of
the Ornstein Uhlenbeck generator $L$ (see for
example~\cite{pecnour_mall} or~\cite{MR2200233} for definitions).  

This approach was pushed further by the same two authors
in~\cite{1196.60034}. Disregarding technicalities, they showed that if
\begin{equation*}
  \left(
    F_n
    ,
    \frac{\left\langle DF_n,-DL^{-1}F_n
      \right\rangle_{\mathfrak{H}}
    -
    \E
    \left\langle DF_n,-DL^{-1}F_n
      \right\rangle_{\mathfrak{H}}
   }{\sqrt{\on{Var} \left\langle DF_n,-DL^{-1} F_n \right\rangle_{\mathfrak{H}}}}
  \right)
\end{equation*}
jointly converges in law to a Gaussian random vector $\left( Z_1,Z_2 \right)$,
it holds that 
\begin{equation}
  \label{eq:22}
  \frac{P(F_n \leq z) - \Phi(z)}{\varphi(F_n)}
  =
  \frac{\E \left[ 1_{(-\infty,z]}(F_n) \right] - \Phi(z)}{\varphi(F_n)} \to \frac{\rho}{3} \Phi^{(3)}(z),
\end{equation}
where $\rho = \E \left[ Z_1 Z_2 \right]$ and $\Phi$ is
the cumulative distribution function of the Gaussian random variable
$Z$ and $\Phi^{(3)}$ denotes its third derivative, thus providing exact asymptotics for the difference $P(F_n \leq z) - \Phi(z)$.

Recently, Nourdin, Peccati and R\'eveillac showed in~\cite{1196.60035}
that the bound~\eqref{eq:15} also has a multidimensional version. It
can still be written as
\begin{equation}
  \label{eq:13}
  \abs{
    \E \left[ g(F_n) \right]
    -
    \E \left[ g(Z) \right]
  }
  \leq
  M(g)
  \,
  \varphi(F_n)
\end{equation}
but now the functionals $F_n$ and the normal $Z$ are $\R^d$-valued and
the function $g$ has to be in $\mathcal{C}^2(\R^d)$ with bounded
first and second derivatives. Moreover, the quantities
$\varphi(F_n)$ are now given by
$\varphi(F_n)=\Delta_{\Gamma}(F_n)+\Delta_C(F_n)$, where
\begin{equation*}
  \Delta_{\Gamma}(F_n)
  =
  \sqrt{\sum_{i,j=1}^d \on{Var} \Gamma_{ij}(F_n)},
\end{equation*}
\begin{equation*}
  \Delta_{C}(F_n)
  =
  \sqrt{\sum_{i,j=1}^d \left( \E \left[ F_{i,n} F_{j,n} \right] - \E
      \left[ Z_i Z_j \right] \right)^2}
\end{equation*}
and $\Gamma_{ij}(F_n) = \left\langle DF_{i,n},-DL^{-1}
  F_{j,n} \right\rangle_{\mathfrak{H}}$.
As every Lipschitz function can be approximated by $\mathcal{C}^2$
functions with bounded derivatives up to order two, \eqref{eq:13}
yields an upper bound for the Wasserstein-distance
(see~\cite{1196.60035}), which, to the knowledge of the author, is the strongest
distance that has been achieved via an approach based on Stein's method (see the
discussion before Theorem~4 in~\cite{MR2453473}). One should note that, using
methods of Malliavin calculus, it is 
however possible to prove that, in several cases, the central limit
theorems implied by the bound~\eqref{eq:13} take place in the total
variation distance (see~\cite[Theorem
5.2]{nourdin_convergence_2012}).
One should also note that another bound for the difference on the left
hand side of~\eqref{eq:15} is given by the maximum of the third and 
fourth cumulants of $F_n$ (see~\cite{bierme_bonami_nourdin_peccati}) and that
this bound is in fact optimal in total variation distance, if the sequence
$(F_n)$ lives in a fixed Wiener chaos (see~\cite{nourdin2013optimal}). 

The main result of this paper is Theorem~\ref{thm:2}, which provides
exact asymptotics for the difference
\begin{equation}
  \label{eq:49}
  \frac{
    \E \left[ g(F_n) \right] - \E \left[ g(Z_n) \right]
  }{
    \varphi(F_n)
  },
\end{equation}
where $(Z_n)$ is a sequence of Gaussian random vectors that has 
the same covariance structure as $(F_n)$. Analogously 
to the one-dimensional case, the random sequences $\left(
  F_n,\widetilde{\Gamma}_{ij}(F_n) \right)_{n \geq 1}$, where
$\widetilde{\Gamma}_{ij}(F_n)$ is a normalized version of
$\Gamma_{ij}(F_n)$, will play a crucial role.

Assuming converging covariances, we are able to obtain an exact and
explicit limit for the quantity~\eqref{eq:49}, where the Gaussian
sequence $(Z_n)$ is replaced by a single Gaussian vector $Z$. This is
Theorem~\ref{thm:5}, which can be seen as a multi-dimensional analogue
to~\eqref{eq:22}. As a by-product, we obtain the optimality  of
$\varphi(F_n)$ for  the Wasserstein  distance $d_W$, by   which we
mean the existence of positive constants~$c_1$ and $c_2$ such
that  
\begin{equation*} 
  c_1 \leq \frac{d_W(F_n,Z)}{\varphi(F_n)} \leq c_2
\end{equation*}
for $n \geq n_0$.
Note that the mere existence of these constants is not hard to
prove. Indeed, a suitable upper bound $c_2$ can always obtained
from~\eqref{eq:13} and by choosing $g$ in~\eqref{eq:13} to depend only
on one coordinate, the problem of finding lower bounds can essentially
be reduced to the one-dimensional findings of~\cite{1196.60034}.

Taking these results a step further, we provide one-term
generalized Edgeworth expansions that speed up the convergence of $(\E
\left[   g(F_n) \right] - \E \left[ g(Z_n) \right])$ (or the
respective sequence with $Z_n$ replaced by $Z$ in the converging
variances case) towards zero. 

As an important special case, we apply Theorems~\ref{thm:2}
and~\ref{thm:5} and their implications to random sequences $(F_n)$,
whose components are elements of some Wiener chaos (that can vary by
component). In this case, the sufficient conditions for our results
simplify substantially and can exclusively be expressed in terms of
contractions of the respective kernels (or even cumulants in the case
of the second chaos). In many cases, the only contractions
one has to look at are those where all kernels are taken from the 
same component of $F_n$, in the spirit of part (B) of the Fourth
Moment Theorem~\ref{thm:7}.

The remainder of the paper is organized as follows. In the preliminary
Section~\ref{s-7}, we introduce the necessary mathematical theory
and gather some results from the existing literature. 
Our main results in a general framemork are presented in
Section~\ref{s-3}. In the following Section~\ref{s-6}, these results
are then specialized to the case where all components of $(F_n)$ are
multiple integrals. We conclude by applying our methods to several
examples, namely step functions, exploding integrals of Brownian
sheets, continuous time Toeplitz quadratic functionals and the
Breuer-Major Theorem.

\section{Preliminaries}
\label{s-7}

\subsection{Metrics for probability measures and asymptotic normality}
\label{s-13}

We fix a positive integer $d$ and denote by $\mathcal{P}(\R^d)$ the
set of all probability measures on $\R^d$. If $X$ is a $\R^d$-valued
random vector, we denote its law by $P_X$.
If $(P_n) \subset \mathcal{P}(\R^d)$ is a sequence of probability
measures, weakly converging to some limit $P$, we can always find an
almost surely converging sequence $(X_n)$ of $\R^{d}$-valued random
vectors, such that $X_n$ has law $P_n$. This is the well-known
Skorokhod representation theorem, which we will state here for
convenience.

\begin{theorem}[Skorokhod representation theorem, \cite{MR0084897}]
  \label{thm:12}
  Let $(P_n)_{n \geq 0} \subset \mathcal{P}(\R^d)$ be a a sequence of
  probability 
  measures such that $P_n \xrightarrow{\mathcal{L}} P_0$. Then there
  exists a sequence $(X_n)_{n \geq 0}$ of $\R^d$-valued random
  vectors, defined on some common probability space
  $(\Omega^{\ast},\mathcal{F}^{\ast},P^{\ast})$, such that $P_{X_n} =  P_n$ and
  $X_n \to X$ $P$-almost surely.
\end{theorem}

Given a metric $\gamma$ on $\mathcal{P}(\R^d)$, we say that $\gamma$
metrizes the weak convergence on $\mathcal{P}(\R^d)$, if for all $P \in
\mathcal{P}(\R^d)$ and sequences $(P_n) \subseteq \mathcal{P}(\R^d)$
the following equivalence holds:
\begin{equation*}
  \gamma(P_n,P) \to 0
  \quad \Leftrightarrow \quad
  P_n \xrightarrow{\mathcal{L}} P.
\end{equation*}
Two prominent examples are the
Prokhorov metric $\rho$ 
and the Fortet-Mourier metric $\beta$, defined by 
\begin{equation*}
  \rho(P,Q) =
  \inf \left\{
    \varepsilon > 0 \colon P(A) \leq Q(A^{\varepsilon}) + \varepsilon
    \quad \text{for every Borel set $A \subseteq \R^d$}
  \right\},
\end{equation*}
 and
\begin{equation*}
  \beta(P,Q) =
  \sup \left\{ \,
    \abs{
      \int f \diff{\, (P-Q)}
    }
    \colon
    \norm{f}_{\infty} + \norm{f}_L \leq 1
    \,
  \right\}.
\end{equation*}
Here, $A^{\varepsilon} = \{ x \colon
\norm{x-y} \leq \varepsilon$ for some $y \in A \}$, $\norm{\cdot}$ is the
$\varepsilon$-hull with respect to the Euclidean norm and $\norm{\cdot}_L$
denotes the Lipschitz seminorm. 
For double sequences of probability measures whose elements are asymptotically close
with respect to one of 
these two metrics, a result similar to the Skorokhod
Representation Theorem~\ref{thm:12} holds.

\begin{theorem}[\cite{MR1932358}, Th. 11.7.1]
  \label{thm:10}
  Let $(P_n)_{n \geq 1}$, $(Q_n)_{n \geq 1} \subseteq
  \mathcal{P}(\R^d)$ be two sequences of probability measures. Then
  the following three conditions are equivalent.
  \begin{enumerate}[a)]
  \item $
    \beta \left( P_n,Q_n \right) \to 0
    $
  \item $
    \rho \left( P_n,Q_n \right) \to 0
    $
  \item There exist two sequences $(X_n)$ and
    $(Y_n)$ of $\R^d$-valued random vectors, defined
    on some common probabilty space
    $(\Omega^{\ast},\mathcal{F}^{\ast},P^{\ast})$, such that
    $P_{X_n} = P_n$ and $P_{Y_n} = Q_n$ for $n
    \geq 1$ and $X_n - Y_n \to 0$ $P$-almost surely.
  \end{enumerate}
\end{theorem}

Note that the Skorokhod Representation Theorem~\ref{thm:12} is not a
simple corollary of Theorem~\ref{thm:10}. Also, the distances $\beta$
and $\rho$ can not easily be replaced by other metrics
(see~\cite[p.418]{MR1932358} for details and counterexamples). Theorem~\ref{thm:10} is the 
motivation for our following definition of asymptotically close
normality. 

\begin{definition}
  \label{def:1}
  Let $(X_n)$ be a sequence of $\R^d$-valued
  random vectors with finite first and second moments. We 
  say that $(X_n)$ is \emph{asymptotically close to
    normal} (in short: \emph{ACN}), if 
  \begin{equation*}
      \beta(P_{X_n},P_{Z_n}) \to 0,
  \end{equation*}
  where the probabilty measures $P_{Z_n}$ are laws of $d$-dimensional
  Gaussian random variables $Z_n$, 
  whose first and second moments coincide with those of $X_n$.
\end{definition}

Note that we consider (almost
surely) constant random vectors as being ``degenerated''
Gaussians. Thus, by the above definition, all sequences of random
vectors whose second moments eventually vanish are ACN.
By Theorem~\ref{thm:10}, we could of course replace the Fortet-Mourier
metric $\beta$ with the Prokhorov metric $\rho$. It is clear that if
$(X_n)$ is ACN, the same is true for any of its components
$(X_{i,n})$. Furthermore, if all first and second moments of $(X_n)$
converge (or, as a special case, are equal), being ACN is equivalent to
converging in law to a Gaussian random variable $Z$ (with the limiting moments
as parameters). Indeed, the 
triangle inequality gives
\begin{equation*}
  \rho(P_{X_n},P_Z) \leq \rho(P_{X_n},P_{Z_n}) + \rho(P_{Z_n},P_Z).
\end{equation*}

We will use the following asymptotic notation for two positive
sequences $(a_n)$ and $(b_n)$ throughout the text. We write$(a_n)  \preccurlyeq
(b_n)$, if there exists a positive constant $c$ such that $a_n \leq c \, b_n$
for $n \geq n_0$ and 
$(a_n) \asymp (b_n)$, if $(a_n) 
\preccurlyeq (b_n)$ and $(b_n) \preccurlyeq (a_n)$ holds. For brevity, we often  
drop the braces and just write $a_n \preccurlyeq b_n$, $a_n \asymp b_n$, etc.  

\subsection{Hermite polynomials, integration by parts and the transformation $U_{g,C}$}
\label{s-4}

Fix a positive integer $d$. Elements of the set $\N_0^d$, where $\N_0 = \N
\cup \{ 0 \}$, will be called ($d$-dimensional) multi-indices. We
define $\abs{\alpha}=\sum_{i=1}^d \alpha_i$ and call this sum the
order of $\alpha$. For a $d$-dimensional vector $x = (x_1,\ldots,x_d)
\in \R^d$, we define $x^{\alpha} := \prod_{i=1}^d x_i^{\alpha_i}$. 
Multi-indices of order one will sometimes be denoted by $e_i$,
where the index $i$ marks the position of the non-zero 
entry. Thus, for example, $x^{e_i}=x_i$.
It is 
clear that 
every multi-index $\alpha$ can be written as a sum of $\abs{\alpha}$
multi-indices $l_1,\ldots,l_{\abs{\alpha}}$ of order one, and that
this sum is unique up to the order of the summands. We will call the
set $\{ l_1,\ldots,l_{\abs{\alpha}}\}$  of these multi-indices the
\emph{elementary decomposition} of $\alpha$. 
For example, the elementary decomposition for the multi-index
$(2,0,1)$ is $\{ (1,0,0), (1,0,0), (0,0,1) \}$.

For any multi-index $\alpha$, the multidimensional \emph{Hermite
  polynomials} $H_{\alpha}(x,\mu,C)$ are defined by
\begin{equation}
  \label{eq:61}
  H_{\alpha}(x,\mu,C)
  =
  \frac{
    (-1)^{\abs{\alpha}}
    \partial_{\alpha} \phi_d(x,\mu,C)
  }{
    \phi_d(x,\mu,C)
  },
\end{equation}
where $\phi_d(x,\mu,C)$ denotes the density of a $d$-dimensional
Gaussian random variable with mean vector $\mu$ and positive definite
covariance matrix $C$ (see for example~\cite[Section 5.4]{MR907286}). Note that in
the case $\mu=0$ and $d=C=1$, this
definition yields the well known one-dimensional Hermite polynomials.
The first few multidimensional Hermite polynomials are given by
$H_0(x,\mu,C) = 1$, 
\begin{align*}
  H_{e_i}(x,\mu,C) &= \sum_{k=1}^d c_{ik} \, (x_k-\mu_k)
  \intertext{and}
  H_{e_i+e_j}(x,\mu,C) &= H_{e_i}(x,\mu,C) \, H_{e_j}(x,\mu,C) - c_{ij},
\end{align*}
where $1 \leq i,j \leq d$ and $C^{-1} = (c_{ij})_{1 \leq i,j \leq d}$ denotes
the inverse of $C$. 

The polynomial $H_{\alpha}(x,\mu,C)$ is of order
$\abs{\alpha}$ and one can show that for fixed $\mu$ and $C$, the
family~$\{ H_{\alpha}(x,\mu,C) \colon \alpha \in \N_0^d \}$ is
orthogonal in $L^2(\R^d,\phi_d(x,\mu,C))$. Furthermore, by integration
by parts, we obtain the identity
\begin{equation}
  \label{eq:64}
    \E \left[
      \partial_{i} f(Z) \, H_{\alpha}(Z,\mu,C)
    \right]
    =
    \E \left[
      f(Z)
      \,
      H_{\alpha+e_i}(Z,\mu,C)
    \right],
\end{equation}
where $1 \leq i \leq d$, $\alpha \in \N_0^d$, $f \in \on{Lip}(\R^d)$
with at most polynomial growth and $Z$ is a Gaussian random 
variable with mean $\mu$ and covariance matrix $C$. Note that the left
hand side of~\eqref{eq:64} is well-defined by Rademacher's
theorem, which guarantees the differentiability of the Lipschitz continuous
function $f$ almost everywhere.
We will also make use of another integration by parts formula, which
can be verified by direct calculation, namely
\begin{equation}
  \label{eq:10}
  \E \left[ f(Z) \, Z_i \right]
  =
  \sum_{j=1}^d \E \left[ Z_i Z_j \right] \E \left[ \partial_j f(Z) \right],
\end{equation}
where $1 \leq i \leq d$, $f$ as above and $Z$ a $d$-dimensional
Gaussian random variable (with possibly singular covariance
matrix). 

For a given Lipschitz function $g \colon \R^d \to \R$ and a positive
semi-definite and symmetric matrix $C$ of dimension $d \times d$, we
define $U_{g,C} \colon \R^d \to \R$ by
\begin{equation}
  \label{eq:9}
  U_{g,C}(x)
  =
  \int_a^b
  \frac{\upsilon'(t)}{\upsilon(t)}
  \left(
  \E \left[
    g(N)
    \right]
    -
    \E \left[ 
    g \left(
      \upsilon(t) x + \sqrt{1-\upsilon^2(t)} N
    \right)
  \right]
  \right)
  \diff{t},  
\end{equation}
where $N$ is a $d$-dimensional centered Gaussian random variable with covariance 
$C$, $-\infty \leq a < b \leq \infty$ and $\upsilon \colon (a,b) \to
(0,1)$ is a diffeomorphism with $\lim_{t \to a+} \upsilon(t) = 0$ (and
therefore $\lim_{t \to b-} \upsilon(t) = 1$). From the change of variables
$\upsilon(t)=s$,
we see that $U_{g,C}$ does not depend on the particular choice of  
$\upsilon$ and by choosing $\upsilon(t)=\me^{-t}$ on the interval $(0,\infty)$, we can write
\begin{equation*}
  U_{g,C}(x) = \int_0^{\infty} \big( P_tg(x) - P_{\infty}g(x) \big) \diff{t},
\end{equation*}
where $P_t g(x) = \E \left[ g \left( \me^{-t}x + \sqrt{1-\me^{-2t}} N
  \right) \right]$, $P_{\infty} g(x) := \lim_{t \to \infty} P_t
g(x) = \E \left[ g(N) \right]$ and $N$ is defined as above. The
operators $P_t$ form the well-known Ornstein-Uhlenbeck semigroup on $\R^d$
(see~\cite[Chapter 1]{pecnour_mall} for details).

Before stating some properties of $U_{g,C}$, let us introduce some
more notation. If $f
\in C^k(\mathbb{R}^d)$ and $\alpha \in \N_0^d$ is a multi-index with
elementary decomposition $\{ l_1,l_2,\ldots,l_{\abs{\alpha}} \}$, we
write  $\partial_{\alpha} f$ or $\partial_{l_1l_2\cdots l_{\abs{\alpha}}} f$ instead
of the more cumbersome $\frac{\partial^{\abs{\alpha}} f}{\partial x_{l_1} \partial x_{l_2}
\cdots \partial x_{l_{\abs{\alpha}}}}$.

\begin{lemma}
  \label{lem:10}
  Let $g \colon \R^d \to \R$ be a Lipschitz-function with at most
  polynomial growth. Furthermore, let $Z$ be a centered,
  $d$-dimensional Gaussian random variable with covariance matrix $C$
  and define $U_{g,C}$  via~\eqref{eq:9}.  Then the following is true.
\begin{enumerate}[a)]
\item The function $U_{g,C}$ satisfies the multidimensional Stein 
equation
\begin{equation*}
  \left\langle C, \on{Hess} U_{g,C}(x) \right\rangle_{\text{H.S.}}
  -
  \left\langle
    x, \nabla U_{g,C}(x)
  \right\rangle_{\R^d}
  =
  g(x)
  -
  \E \left[ g(Z) \right].
\end{equation*}
\item If $g$ is $k$-times differentiable with bounded
  derivatives up to order $k$, the same is true for $U_{g,C}$. In this
  case, for any $\alpha \in \N_0^d$ with $\abs{\alpha}\leq k$, the
  derivatives are given by 
  \begin{equation}
    \label{eq:65}
    \partial_{\alpha} U_{g,C}(x) =
    \int_a^b
    \upsilon'(t) \, \upsilon^{\abs{\alpha}-1}(t)
    \,
    \E \left[
      \partial_{\alpha} g \left(
        \upsilon(t) x
        +
        \sqrt{1-\upsilon^2(t)} N
      \right)
    \right]
    \diff{t}
  \end{equation}
  and it holds that
  \begin{equation}
  \label{eq:76}
  \abs{\partial_{\alpha} U_{g,C}(x)}
  \leq
  \frac{
      \norm{\partial_{\alpha} \, g}_{\infty}
    }{\abs{\alpha}}
  \end{equation}
  and
    \begin{equation}
    \label{eq:54}
    \E \left[ \partial_{\alpha} U_{g,C}(Z) \right]
    =
    \frac{1}{\abs{\alpha}}
    \E \left[ \partial_{\alpha} g(Z) \right].
  \end{equation}
\end{enumerate}  
\end{lemma}

\begin{proof}
  For a proof of part a) see~\cite{1196.60035}.  
  Repeated differentiation under the integral
  sign (the first one being justified by the Lipschitz property of $g$) shows
  formula~\eqref{eq:65}, of which the bound~\eqref{eq:76} 
  is an immediate consequence. To show~\eqref{eq:54}, we again use
  formula~\eqref{eq:65} and the fact that $\upsilon(t)Z +
  \sqrt{1-\upsilon^2(t)}N$ has the same law as $Z$. This gives
  \begin{align*}
    \E \left[ \partial_{\alpha} U_{g,C}(Z) \right]
    &=
    \int_a^b
    \upsilon'(t) \upsilon^{\abs{\alpha}-1}(t)
    \E \left[
      \partial_{\alpha} g \left(
        \upsilon(t) Z + \sqrt{1-\upsilon^2(t)}N
      \right)
    \right]
    \diff{t}
    \\ &=
    \int_a^b
    \upsilon'(t) \upsilon^{\abs{\alpha}-1}(t)
    \diff{t}
    \E \left[
      \partial_{\alpha} g \left( Z \right)
    \right]
    \\ &=
    \frac{1}{\abs{\alpha}}
    \E \left[
      \partial_{\alpha} g \left( Z \right)
    \right].
  \end{align*}
\end{proof}

\subsection{Isonormal Gaussian processes and Wiener chaos}
\label{s-14}

For a detailed discussion of the notions introduced in this section,
we refer to~\cite{pecnour_mall} or~\cite{MR2200233}.

Fix a real separable Hilbert space
$\mathfrak{H}$ and a family $X = \{ X(h) \colon h \in \mathfrak{H}\}$
of centered Gaussian random variables, defined on some probability
space $(\Omega,\mathcal{F},P)$, such that the isometry property $\E
\left[ X(g) X(h) \right] = \left\langle g,h
\right\rangle_{\mathfrak{H}}$ holds for $g,h \in \mathfrak{H}$. Such
a family $X$ is called an \emph{isonormal Gaussian process} over
$\mathfrak{H}$. Without loss of generality, we assume that the
$\sigma$-field $\mathcal{F}$ is generated by $X$.
For $q \geq 1$, we denote the $q$th tensor product of $\mathfrak{H}$ by
$\mathfrak{H}^{\otimes q}$ and the $q$th symmetric tensor product of
$\mathfrak{H}$ by  $\mathfrak{H}^{\odot q}$. Furthermore, we define
$\mathcal{H}_q$, the \emph{Wiener chaos of order $q$} (with respect to
$X$), to be the closed linear subspace of  $L^2(\Omega,\mathcal{F},P)$
generated by the set $\{ H_q(X(h)) \colon h \in \mathfrak{H}, \, 
\norm{h}_{\mathfrak{H}}=1 \}$, where $H_q(x)=H_q(x,1)$ denotes the $q$th
Hermite polynomial, defined by~\eqref{eq:61}. The mapping $I_q(h^{\otimes
  q}) = q! H_q(X(h))$, where~$\norm{h}_{\mathfrak{H}}=1$, can be extended to a
  linear isometry between the symmetric tensor product $\mathfrak{H}^{\odot q}$, equipped with the
modified norm $\sqrt{q!} \, \norm{\cdot}_{\mathfrak{H}^{\otimes q}}$, and
the $q$th Wiener chaos $\mathcal{H}_q$.
Wiener chaoses of different orders are orthogonal. More precisely, if $f_i \in
\mathfrak{H}^{\odot   q_i}$ and $f_j \in \mathfrak{H}^{\odot q_j}$ for 
$q_i,q_j \geq 1$ it holds that
\begin{equation}
  \label{orthogonality}
  \E \left[ I_{q_i}(f_i) \, I_{q_j}(f_j) \right]
  =
  \begin{cases}
    q_i! \left\langle f_i,f_j \right\rangle_{\mathfrak{H}}
    &\quad \text{if $q_i=q_j$}
    \\
    0
    &\quad \text{if $q_i \neq q_j$}.
  \end{cases}
\end{equation}
Furthermore, the \emph{Wiener chaos decomposition} tells us that the
space $L^2(\Omega,\mathcal{F},P)$ can be decomposed into the infinite 
orthogonal sum of the $\mathcal{H}_q$. As a consequence, any
square-integrable random variable $F \in L^2(\Omega,\mathcal{F},P)$
can be written as
\begin{equation}
  \label{chaosexp}
  F = \E \left[ F \right] + \sum_{q=1}^{\infty} I_{q}(f_q).
\end{equation}
where the kernels $f_q \in \mathfrak{H}^{\odot q}$ are uniquely defined. This
identity is called the \emph{chaos expansion} of $F$.

If $\{ \psi_k \colon k \geq 1 \}$ is a complete
orthonormal system in $\mathfrak{H}$, $f_i \in \mathfrak{H}^{\odot
  q_i}$, $f_j \in \mathfrak{H}^{\odot q_j}$ and $r \in \{0,\ldots,q_i
\land q_j \}$, the \emph{contraction} $f_i \otimes_r f_j$
of $f_i$ and $f_j$ of order $r$ is the element of $\mathfrak{H}^{\otimes
  (q_i+q_j-2r)}$ defined by
\begin{equation}
  \label{eq:18}
  f_i \otimes_r f_j
  =
  \sum_{l_1, \ldots,l_r=1}^{\infty}
  \left\langle
    f_i,\psi_{l_1} \otimes \cdots \otimes \psi_{l_r}
  \right\rangle_{\mathfrak{H}^{\otimes r}}
  \otimes
  \left\langle
    f_j,\psi_{l_1} \otimes \cdots \otimes \psi_{l_r}
  \right\rangle_{\mathfrak{H}^{\otimes r}}.
\end{equation}
The contraction $f_i \otimes_r f_j$ is not necessarily
symmetric. We denote its canonical symmetrization by $f_i \widetilde{\otimes}_r f_j \in
\mathfrak{H}^{\odot q_i+q_j-2r}$. Note that $f_i \otimes_0 f_j$ is
equal to the usual tensor product $f_i \otimes f_j$ of $f_i$ and
$f_j$. Furthermore, if $q_i=q_j$, we have that $f_i \otimes_{q_i} f_j
= \left\langle f_i,f_j \right\rangle_{\mathfrak{H}^{\otimes q_i}}$.
The well known multiplication formula 
\begin{equation}
  \label{prodform}
  I_{q_i}(f_i) \, I_{q_j}(f_j)
  =
  \sum_{r=0}^{q_i \land q_j}
  \beta_{q_i,q_j}(r) \,
  I_{q_i+q_j-2r}(f_i \widetilde{\otimes}_r f_j),
\end{equation}
where 
\begin{equation}
  \label{eq:86}
  \beta_{a,b}(r) = r! \binom{a}{r} \binom{b}{r},
\end{equation}
gives us the chaos expansion of the product of two multiple integrals.

When $\mathfrak{H} = L^2(A,\mathcal{A},\nu)$, where $(A,\mathcal{A})$
is a Polish space, $\mathcal{A}$ is the associated Borel
$\sigma$-field and the measure $\mu$ is positive, $\sigma$-finite and
non-atomic, one can identify the
symmetric tensor product $\mathfrak{H}^{\odot q}$ with the Hilbert
space $L^2_s(A^q,\mathcal{A}^q,\nu^{\otimes q})$, which is defined as
the collection of all $\nu^{\otimes q}$-almost everywhere symmetric functions an
$A^q$, that are 
square-integrable with respect to the product measure $\nu^{\otimes
  q}$. In this case, the random variable $I_q(h)$, $h \in
\mathfrak{H}^{\odot q}$, coinicides with the multiple Wiener-It\^o
integral of order $q$ of $h$ with respect to the Gaussian measure $B
\mapsto X(1_B)$, where $B \in \mathcal{A}$ and $\nu(A) <
\infty$. Furthermore, the contraction~\eqref{eq:18} can be written as
\begin{multline}
  \label{eq:20}
  (f_i \otimes_r f_j) (t_1,\ldots,t_{q_i+q_j-2r})
  \\ =
  \int_{A^r}
  f_i(t_1,\ldots,t_{q_i-r},s_1,\ldots,s_r)
  f_j(t_{q_i-r+1},\ldots,t_{q_i+q_j-2r},s_1,\ldots,s_r)
  \\
  \diff{\nu(s_1)} \ldots \diff{\nu(s_r)}.
\end{multline}

\subsection{Operators from Malliavin calculus}
\label{s-10}

In this section, we introduce the operators $D$, $L$ and $L^{-1}$
from Malliavin calculus, which will appear in the statements of our
main results. This exposition is by no means complete, most notably, we
do not introduce the divergence operator. Again, we refer
to~\cite{pecnour_mall} or~\cite{MR2200233} for a full discussion.

If $\mathcal{S}$ is the set of all cylindrical random variables of
the type
\begin{equation*}
  F = g(X(h_1),\ldots,X(h_k)),
\end{equation*}
where $k \geq 1$, $h_i \in \mathfrak{H}$ for $1 \leq i \leq k$ and
$g \colon \R^k \to \R$ is an infinitely differentiable function with
compact support, the Malliavin derivative $DF$ with respect to $X$
is the element of $L^2(\Omega,\mathfrak{H})$ defined by
\begin{equation*}
  DF = \sum_{i=1}^k \partial_i (X(h_1),\ldots,X(h_k)) h_i.
\end{equation*}
Iterating this procedure, we obtain higher derivatives $D^mF$ for any
$m \geq 2$, which are elements of $L^2(\Omega,\mathfrak{H}^{\odot
  m})$.
For $m,p \geq 1$, $\mathbb{D}^{m,p}$ denotes the closure of
$\mathcal{S}$ with respect to the norm $\norm{\cdot}_{m,p}$, which is
defined by
\begin{equation*}
  \norm{F}_{m,p}^p = \E \left[ \abs{F}^p \right]
    +
    \sum_{i=1}^m
    \E \left[ \norm{D^i F}_{\mathfrak{H}^{\otimes i}}^p \right].
\end{equation*}
If $\mathfrak{H}=L^2(A,\mathcal{A},\nu)$, with $\nu$ non-atomic, the
Malliavin derivative of a random variable $F$ having the chaos
expansion~\eqref{chaosexp} can be identified with the element of
$L^2(A \times \Omega)$ given by
\begin{equation}
  \label{eq:84}
  D_t F = \sum_{q=1}^{\infty} qI_{q-1}(f_q(\cdot,t)), \quad t \in A.
\end{equation}

The \emph{Ornstein-Uhlenbeck generator} $L$ is defined by $L =
\sum_{q=0}^{\infty} -qJ_q$. Here, $J_q$ denotes the orthogonal
projection onto the $q$th Wiener chaos. The domain of $L$ is
$\mathbb{D}^{2,2}$. Similarly, we define its pseudo-inverse $L^{-1}$
by $L^{-1} = \sum_{q=1}^{\infty} - \frac{1}{q} J_q$. This
pseudo-inverse is defined on $L^2(\Omega)$ and for any $F \in
L^2(\Omega)$ it holds that $L^{-1}F$ lies in the domain of $L$. The
name pseudo-inverse is justified by the relation
\begin{equation*}
  L L^{-1} F = F - \E \left[ F \right],
\end{equation*}
valid for any $F \in L^2(\Omega)$.

\subsection{Cumulants}
\label{s-16}

Recall the multi-index notation introduced in the first paragraph of
Section~\ref{s-4}. 

Let $F=(F_1,\ldots,F_d)$ be a $\R^d$-valued random vector. For a 
multi-index $\alpha$, we set $F^{\alpha} = \prod_{k=1}^d
F_k^{\alpha_k}$ and, with slight abuse of notation, $\abs{F}^{\alpha} =
\prod_{k=1}^d \abs{F_k}^{\alpha_k}$. The moments $\mu_{\alpha}(F)$ of
$F$ of order $\abs{\alpha}$ are then defined by $\mu_{\alpha}(F) = \E \left[
  F^{\alpha} \right]$, provided that the expectation on the right hand
side is finite. Analogously, one defines the absolute moments 
$\mu_{\alpha}(\abs{F})$.
We denote by $\phi_{F}(t)=\E \left[ \exp \left( \mathrm{i}
    \left\langle t,F  \right\rangle_{\R^d} \right) \right]$ the
characteristic function 
of $F$. If $\mu_{\alpha}(\abs{F})<\infty$, the \emph{joint cumulant}
$\kappa_\alpha(F)$ of order $\abs{\alpha}$ of $F$ is defined by
\begin{equation*}
  \kappa_{\alpha}(F) = (-\mathrm{i})^{\abs{\alpha}} \partial_{\alpha} \log
  \phi_F(t) |_{t=0}.
\end{equation*}
Given all joint cumulants $\kappa_{\alpha}(F)$ up to some order exist,
we can compute the moments up to the same order by Leonov and Shiryaev's
formula (see \cite[Proposition 3.2.1]{MR2791919})
\begin{equation}
  \label{eq:31}
  \mu_{\alpha}(F)
  =
  \sum_{\pi}
  \kappa_{b_1}(F) \kappa_{b_2}(F) \cdots \kappa_{b_m}(F),
\end{equation}
where the sum is taken over all partitions $\pi=\{B_1,\ldots,B_m \}$
of the elementary decomposition of $\alpha$ and the multi-indices
$b_k$ are defined by $b_k= \sum_{l_j \in B_k} l_j$. For example, if
$1 \leq i,j,k \leq d$, we get
$\mu_{e_i}(F) = \kappa_{e_i}(F)$,
$\mu_{e_i+e_j}(F)
  =
  \kappa_{e_i+e_j}(F) + \kappa_{e_i}(F) \kappa_{e_j}(F) 
$
and
\begin{multline*}
  \mu_{e_i+e_j+e_k}(F)
  =
  \kappa_{e_i+e_j+e_k}(F)
  \\
  + \kappa_{e_i}(F) \kappa_{e_j+e_k}(F)
  + \kappa_{e_j}(F) \kappa_{e_i+e_k}(F)
  + \kappa_{e_k}(F) \kappa_{e_i+e_j}(F)
  \\
  + \kappa_{e_i}(F) \kappa_{e_j}(F) \kappa_{e_k}(F)
\end{multline*}
for the moments of order one, two and three, respectively. Note that
if $F$ is centered, all moments of order less than four coincide with
the respective cumulants.

\subsection{Generalized Edgeworth expansions}
\label{s-5}

Let $F_1$ and $F_2$ be two $\R^d$-valued random vectors with
finite absolute moments up to some order $m \in \N_0 \cup
\{\infty\}$ and consider the problem of approximating $F_1$ in terms
of $F_2$. The classical Edgeworth expansion provides such an 
approximation in terms of formal ``moments'', which we will now describe. 

For every multi-index $\alpha$ of order at most $m$, we
define formal ``cumulants'' $\widetilde{\kappa}_{\alpha}(F_1,F_2)$ by 
$\widetilde{\kappa}_{\alpha}(F_1,F_2) = \kappa_{\alpha}(F_1) -
\kappa_{\alpha}(F_2)$ and use Shiryaev's formula~\eqref{eq:31} to
define corresponding formal ``moments''
$\widetilde{\mu}_{\alpha}(F_1,F_2)$. Two things are important to note
at this point. In general, $\widetilde{\mu}_{\alpha}(F_1,F_2) \neq
\mu_{\alpha}(F_1) - \mu_{\alpha}(F_2)$ and the collection $\{
\widetilde{\kappa}_{\alpha}(F_1,F_{2}) \colon \abs{\alpha} \leq m \}$
can not be represented as cumulants associated with some random
variable. If we now assume that $F_1$ and $F_2$ both have densities, 
say $f_1$ and $f_2$, the classical Edgeworth
expansion of order $m$ for the density $f_1$ then reads 
\begin{equation}
  \label{eq:56}
  f_1(x) \sim f_2(x) +
  \sum_{1 \leq \abs{\alpha} \leq m}
  \frac{(-1)^{\abs{\alpha}}}{\abs{\alpha}!}
  \
  \widetilde{\mu}_{\alpha}(F_1,F_2)
  \partial_{\alpha} f_2(x).
\end{equation}

In the most prominent example where this is the case, $F_1$ is a
normalized sum of iid random variables and $F_2$ is Gaussian. In this
case, the Edgeworth expansion can be used to improve the speed of
convergence in the classical central limit theorem.  For details, we
refer to~\cite{MR1145237}, \cite[Chapter5]{MR907286} and~\cite{MR855460}.

For our framework, however, the classical Edgeworth expansion is too
rigid, as we cannot assume the existence of
(smooth) densities. Therefore, instead of expanding the density $f_1$
in terms of $f_2$ and its derivatives, we pass to the distributional
operators 
$g  \mapsto \E \left[ g(F_1) \right]$ and $g \mapsto \E \left[ g(F_2)
\right]$, defined on the space of infinitely differentiable functions 
with compact support. The expansion~\eqref{eq:56} becomes 
\begin{equation}
  \label{eq:58}
  \E \left[ g(F_1) \right]
  \sim
  \E \left[ g(F_2) \right]
  +
  \sum_{1 \leq \abs{\alpha} \leq m}
  \frac{
    \widetilde{\mu}_{\alpha}(F_1,F_2)
  }{
    \abs{\alpha}!
  }
  \,
  \E \left[ \partial_{\alpha}g(F_2) \right].
\end{equation}
Note that in the case of existing smooth densities it holds that
$\E \left[ g(F_1) \right] = \int_{-\infty}^{\infty} g(x) f_1(x)
\diff{x}$
and, by integration by parts,
\begin{equation*}
  \E \left[ \partial_{\alpha} g(F_2) \right]
  =
  \int_{-\infty}^{\infty} \partial_{\alpha} g(x) f_2(x) \diff{x}
  =
  (-1)^{\abs{\alpha}}
  \int_{-\infty}^{\infty} g(x) \partial_{\alpha} f_2(x) \diff{x},
\end{equation*}
so that~\eqref{eq:58} is obtained in a natural way from~\eqref{eq:56},
by multiplying with the test function $g$ and integrating on both
sides.  

This leads us to the following definition of a generalized Edgeworth
expansion.

\begin{definition}[Generalized Edgeworth expansion]
  \label{def:3}
If $g$ is $m$-times
differentiable and has bounded derivatives up to order $m$,
we define the \emph{generalized $m$th order Edgeworth  
  expansion} $\mathcal{E}_m(F_1,F_2,g)$ of $\E \left[ 
  g(F_1) \right]$ 
around $\E \left[ g(F_2) \right]$ by 
\begin{equation}
  \label{eq:21}
  \mathcal{E}_m(F_1,F_2,g)
  =
  \E \left[ g(F_2) \right]
  +
  \sum_{1 \leq \abs{\alpha} \leq m}
  \frac{
      \widetilde{\mu}_{\alpha}(F_1,F_2)    
  }{\abs{\alpha}!}
  \E \left[ \partial_{\alpha} g(F_2) \right].
\end{equation}  
\end{definition}

If $Z$ is a $d$-dimensional centered normal with covariance
matrix $C$ (the case that we will exclusively consider in the sequel),
formula~\eqref{eq:64} yields
\begin{equation*}
  \mathcal{E}_m(F_1,Z,g)
  =
  \E \left[ g(Z) \right]
  +
  \sum_{1 \leq \abs{\alpha} \leq m}
  \frac{
      \widetilde{\mu}_{\alpha}(F_1,F_2)
  }{\abs{\alpha}!}
  \E \left[ g(Z) H_{\alpha}(Z,C) \right],
\end{equation*}
where the Hermite polynomials $H_{\alpha}(x,C)$ are defined
by~\eqref{eq:61} (recall our convention that we drop the mean as an
argument if it is zero). 

\subsection{Cumulant formulas for chaotic random vectors}
\label{s-12}

When dealing with functionals of an isonormal Gaussian process, their
cumulants can be generalized in terms of Malliavin operators. This
(among other things) is the content of~\cite{MR2606872}
and~\cite{MR2793872} (see  also~\cite[Chapter8]{pecnour_mall}), which
we will summarize here.

Let $F=(F_1,\ldots,F_d)$ be a $\R^d$-valued random vector whose
components are functionals of some isonormal Gaussian process $X$ and
let $l_1,l_2,\ldots$ be a sequence of $d$-dimensional multi-indices of
order one.
If $F^{l_1} \in \mathbb{D}^{1,2}$, we set $\Gamma_{l_1}(F) =
F_i$. Inductively, if $\Gamma_{l_1,l_2,\ldots,l_k}(F)$ is a
well-defined element of $L^2(\Omega)$ for some $k \geq 1$, we define
\begin{equation}
  \label{eq:55}
  \Gamma_{l_1,\ldots,l_{k+1}}(F) =
  \left\langle
     DF_{l_{k+1}}
    ,
    -D L^{-1} \Gamma_{l_1,\ldots,l_k}(F)
  \right\rangle_{\mathfrak{H}}.
\end{equation}
The question of existence is answered by the following lemma.

\begin{lemma}[Noreddine, Nourdin~\cite{MR2793872}]
  \label{lem:3}
  With the notation as above, fix an integer $j \geq 1$ and assume
  that $F_i \in \mathbb{D}^{j,2^j}$ for $1 \leq i \leq d$. Then it
  holds that $\Gamma_{l_1,\ldots,l_k}(F)$ is a well-defined element of
  $\mathbb{D}^{j-k+1,2^{j-k+1}}$ for all $k=1,\ldots,j$. In
  particular, $\Gamma_{l_1,\ldots,l_j}(F) \in \mathbb{D}^{1,2} \subset
  L^2(\Omega)$ and the quantity $\E \left[ \Gamma_{l_1,\ldots,l_j}(F)
  \right]$ is well-defined and finite. 
\end{lemma}

Using these random elements, we can now state a formula for the
cumulants of $F$.

\begin{theorem}[Noreddine, Nourdin~\cite{MR2793872}]
  \label{thm:4}
  Let $\alpha$ be a $d$-dimensional multi-index with elementary
  decomposition $\{ l_1,\ldots,l_{\abs{\alpha}} \}$. If $F_i \in
  \mathbb{D}^{\abs{m},2^{\abs{m}}}$ for $1 \leq i \leq d$, then
  \begin{equation}
    \label{eq:60}
    \kappa_\alpha(F) = \sum_{\sigma} 
    \E \left[
      \Gamma_{l_1,l_{\sigma(2)},l_{\sigma(3)},\ldots,l_{\sigma(\alpha)}}(F)
    \right],
  \end{equation}
  where the sum is taken over all permutations $\sigma$ of the set $\{
  2,3,\ldots,\abs{\alpha} \}$.
\end{theorem}

We again stress that -- as the labeling of the elementary
decomposition is arbitrary -- we can freely choose the fixed first
element $l_1$. For the case $d=1$, this formula has been proven
in~\cite{MR2606872}. 

To simplify notation, we will frequently write $\Gamma_{i_1 i_2\cdots
  i_k}(F)$ instead of the more cumbersome
$\Gamma_{e_{i_1},e_{i_2},\ldots,e_{i_k}}(F)$. For example, the random variable 
$\Gamma_{e_1,e_2}(F) = \left\langle DF_1,-DL^{-1}F_2
\right\rangle_{\mathfrak{H}}$ will also be denoted by $\Gamma_{12}(F)$.

If all components of $F$ are elements
of (possibly different) Wiener chaoses, formula~\eqref{eq:60} can be
stated in terms of contractions. We state two special cases here
and refer to Noreddine and Nourdin~\cite{MR2793872} for a general
formula.
As a first special case, assume that $F_i=I_{q_i}(f_i)$ where $q_i
\geq 1$ and $f_i \in \mathfrak{H}^{\odot q_i}$ for $1 \leq i \leq
d$. In this case, for $1 \leq i,j,k \leq d$, the third-order cumulants
are given by 
\begin{equation}
  \label{eq:80}
  \kappa_{e_i+e_j+e_k} =
  \begin{cases}
    c \left\langle f_i \widetilde{\otimes}_r f_j,f_k
    \right\rangle_{\mathfrak{H}^{\otimes q_k}}
    \quad
    &\text{if $r := \frac{q_i+q_j-q_k}{2} \in \{1,2,\ldots,q_i \land q_j
      \}$,}
    \\
    0
    \quad
    &\text{otherwise,}
  \end{cases}
\end{equation}
where $c$ is some positive constant depending on the
chaotic orders $q_i$, $q_j$ and $q_k$.
As a second special case, assume that the components $F_i$ are all of
the form $F_i = I_2(f_i)$, with $f_i \in \mathfrak{H}^{\odot 2}$ for
$1 \leq i\leq d$. In this case, for any $\alpha \in \N_0^d$ with
$\abs{\alpha} \geq 2$ it holds that
\begin{equation}
  \label{eq:37}
  \kappa_{\alpha}(F)
  =
  2^{\abs{\alpha}-1}
  \sum_{\sigma}
  \left\langle
    ( \cdots (
    f_{i_1} \widetilde{\otimes}_1 f_{i_{\sigma(1)}}
    )
    \widetilde{\otimes}_1 f_{i_{\sigma(2)}}
    )
    \ldots
    )
    \widetilde{\otimes}_1
    f_{i_{\sigma(\abs{\alpha}-1)}}
    ,
    f_{i_{\sigma(\abs{\alpha})}}
  \right\rangle_{\mathfrak{H}^{\otimes 2}},
\end{equation}
where the sum is taken over all permutations $\sigma$ of the set $\{
2,3,\ldots,\abs{\alpha} \}$ and the indices
$i_1,\ldots,i_{\abs{\alpha}}$ are defined as follows: If
$\{l_1,\ldots,l_{\abs{\alpha}} \}$ is the elementary decomposition of
$\alpha$, we set $i_j = k$ if $l_j=e_k$,
$j=1,\ldots,\abs{\alpha}$. To illustrate this formula, we
have for example
\begin{equation}
  \label{eq:59}
  \kappa_{(2,1)}(F) = 8 \left\langle f_1 \widetilde{\otimes}_1 f_1,f_2
    \right\rangle_{\mathfrak{H}^{\otimes 2}},
\end{equation}
or, with a different labelling of the elementary decomposition,
\begin{equation}
  \label{eq:63}
  \kappa_{(2,1)}(F) = 8 \left(
    \left\langle f_1 \widetilde{\otimes}_1 f_2,f_1
    \right\rangle_{\mathfrak{H}^{\otimes 2}}
    +
    \left\langle f_2 \widetilde{\otimes}_1 f_1,f_1
    \right\rangle_{\mathfrak{H}^{\otimes 2}}
  \right).
\end{equation}
One can verify by direct computations that the right hand sides
of~\eqref{eq:59} and~\eqref{eq:63} are indeed equal.

\subsection{Limit theorems for vectors of multiple integrals}
\label{s-21}

In this section, we gather two results from the literature which
we will use extensively in the sequel. The first is a version of the so 
called \textit{Fourth Moment Theorem} (see~\cite[Theorem
3]{MR2350573}) for fluctuating covariances, that is based on the
findings in~\cite{MR2118863}, \cite{MR2394845} and~\cite{MR2126978}
for the converging covariance case.

\begin{theorem}[Fourth Moment Theorem,~\cite{MR2350573}]
  \label{thm:7} \hfill
  \begin{enumerate}[(A)]
  \item 
  Let $q \geq 1$ and $(F_n)_{n \geq 1} = (I_q(f_n))_{n \geq 1}$ be a
  sequence of multiple 
  integrals and assume that there exists a constant $M$ such that $\E
  \left[ F_n^2 \right] \leq M$ for $n \geq 1$. Then the following
  conditions are equivalent. 
  \begin{enumerate}[(i)]
  \item \label{item:1} $(F_n)_{n \geq 1}$ is ACN
  \item $\E \left[ F_n^4 \right] - 3 \E \left[ F_n^2 \right]^2 \to 0$
  \item For $1 \leq r \leq q-1$ it holds that \, 
      $\norm{ f_{n} \otimes_r
        f_{n}}_{\mathfrak{H}^{\otimes 2(q-r)}} \to 0$
  \item $\on{Var} \big(
    \Gamma_{11}(F_n) \big) \to 0 $
  \end{enumerate}
  If the variance of $F_n$ converges to some limit $c$, conditions
  (i)-(iv) are equivalent to
  \begin{enumerate}
  \item[(i')] $F_n \xrightarrow{d} Z$, where $Z$ is a centered normal
    with variance $c$.
  \end{enumerate}
\item
    Let $(F_n)_{n \geq 1} = (F_{1,n},\ldots,F_{d,n})_{n \geq 1}$ be a
  random sequence such that $F_{i,n}=I_{q_i}(f_{i,n})$, $q_i \geq 1$, for $1 \leq i
  \leq d$ and the covariances of $(F_n)$ are uniformly bounded. Then
  $(F_n)$ is ACN if and only if $(F_{i,n})$ is ACN for $1 \leq i \leq d$.
    \end{enumerate}
\end{theorem}

Secondly, we will make use of the following central limit theorem for
the case  where one component of the random vectors $F_n$ has a
finite chaos expansion. As this result is an immediate consequence of the
findings in~\cite{MR2350573}, we omit the proof.

\begin{lemma}
  \label{lem:1}
  Let $(F_n)_{n \geq 1} = (F_{1,n},\ldots,F_{d,n})_{n \geq 1}$ be a
  sequence of random vectors such that $F_{i,n}=I_{q_i}(F_{i,n})$ for
  $n \geq 1$ and 
  $1 \leq i \leq d$. Furthermore, let $G_n =
  \sum_{k=1}^M I_k(g_{k,n})$ for $n \geq 1$. If
  \begin{equation}
    \label{eq:25}
    \sum_{i=1}^d
    \sum_{r=1}^{q_i-1}
    \,
    \norm{f_{i,n} \otimes_r f_{i,n}}_{\mathfrak{H}^{\otimes 2(q_i-r)}}
    \to 0
  \end{equation}
  and
  \begin{equation}
    \label{eq:12}
    \sum_{k=2}^M
    \sum_{s=1}^{k-1}
    \norm{g_{k,n} \otimes_s
      g_{k,n}}_{\mathfrak{H}^{\otimes 2(q_k-s)}} \to 0,
  \end{equation}
  then $(F_n,G_n)_{n \geq 1}$ is ACN. 
\end{lemma}

\section{Main results}
\label{s-3}

In this section, for some fixed positive integer $d$, we denote by $(F_n)_{n \geq
  1} = (F_{1,n}, F_{2,n},\ldots,F_{d,n})_{n \geq 1}$  a sequence of centered,
$\R^d$-valued random vectors such that $F_{i,n} \in
\mathbb{D}^{1,4}$ for $1 \leq i \leq d$. We also introduce a normalized
sequence $(\widetilde{\Gamma}_{ij}(F_n))_{n \geq  1}$ for $1 \leq i,j
\leq d$, which is defined by
\begin{equation*}
  \widetilde{\Gamma}_{ij}(F_n) =
  \frac{
    \Gamma_{ij}(F_n) - \E \left[ \Gamma_{ij}(F_n) \right]
  }{
    \sqrt{\on{Var} \Gamma_{ij}(F_n)}
  }
  =
  \frac{
    \Gamma_{ij}(F_n) - \E \left[ F_{i,n} F_{j,n} \right]
  }{
    \sqrt{\on{Var} \Gamma_{ij}(F_n)}
  },
\end{equation*}
where $\Gamma_{ij}(F_n)$ is given by~\eqref{eq:55}. Furthermore, for $1 \leq
i,j \leq d$, let $(Z_n)_{n \geq 1} =
(Z_{1,n},\ldots,Z_{d,n})_{n \geq 1}$ be a centered
sequence of Gaussian random variables such that $Z_n$ has the
same covariance as $F_n$ for $n \geq 1$.
The following crucial identity is the starting point
of our investigations.

\begin{theorem}[\cite{pre05793427}]
  \label{thm:3}
  Let $g \in \mathcal{C}^2(\R^d)$ and $Z$ be a $d$-dimensional normal
  vector with
  covariance matrix $C$. Then, for every $n \geq 1$, it holds that
    \begin{equation}
      \label{eq:62}
      \E \left[ g(F_n) \right] - \E \left[ g(Z) \right]
      =
      \sum_{i,j=1}^d
      \E \left[
        \partial_{ij} U_{g,C}(F_n)
        \,
        \left(
          \Gamma_{ij}(F_n) - C_{ij} 
        \right)
      \right],
    \end{equation}
where $U_{g,C}$ is defined by~\eqref{eq:9}.
\end{theorem}

Identity~\eqref{eq:62} has been derived in~\cite{pre05793427} by the
so called ``smart path method'' and Malliavin calculus. If the covariance
matrix 
$C$ is positive definite, one can give an alternative proof by using
Stein's method (see~\cite[proof of Theorem 3.5]{1196.60035}).

A straightforward application of the Cauchy-Schwarz inequality to
identity~\eqref{eq:62} yields the bound
\begin{equation}
  \label{eq:75}
  \abs{
    \E \left[ g(F_n) \right] - \E \left[ g(Z) \right]
  }
  \leq
  \frac{\sqrt{d}}{2}
  \left(
    \sup_{\abs{\alpha}=2}
    \norm{\partial_{\alpha} g}_{\infty}
  \right)
  \,
  \varphi_C(F_n),
\end{equation}
where $\varphi_C(F_n) = \Delta_{\Gamma}(F_n) + \Delta_C(F_n)$ and the
quantities $\Delta_{\Gamma}(F_n)$ and $\Delta_C(F_n)$, that already
appeared in the Introduction, are defined by
\begin{equation}
  \label{eq:71}
  \Delta_{\Gamma}(F_n)
  =
  \norm{\Gamma(F_n) - \on{Cov}(F_n)}_{H.S.}
  =
    \sqrt{
    \sum_{i,j=1}^d
    \on{Var} \,
      \Gamma_{ij}(F_n)
    }
\end{equation}
and
\begin{equation}
  \label{eq:7}
  \Delta_C(F_n) =
  \norm{\on{Cov}(F_n)-\on{Cov}(Z)}_{H.S.}
  =
    \sqrt{
      \sum_{i,j=1}^d
      \left(\E \left[ F_iF_j \right] - C_{ij} \right)^2
    }.
\end{equation}
Here, $\norm{\cdot}_{H.S.}$ denotes the Hilbert-Schmidt matrix
norm. Note that $\Delta_C(F_n)$ is equal to zero if
and only if $F_n$ has covariance matrix $C$ and that $\Delta_{\Gamma}(F_n)$ is
equal to zero if $F_n$ is Gaussian. The latter follows from the fact that 
$\abs{\Gamma_{ij}(F_n)}$ is constant if $F_{i,n}$ and $F_{j,n}$
are Gaussian, which can be seen by applying the bound~\eqref{eq:75} to the
vector $(F_{i,n},F_{j,n})$ and a centered Gaussian vector $(Z_1,Z_2)$ with the
same covariance.  

Assume now that $\varphi_C(F_n)$ converges to zero. For the
one-dimensional case $d=1$, an adaptation of the
arguments in~\cite{1196.60034} provides conditions under which the ratio
\begin{equation*}
  \frac{
    \E \left[ g(F_n) \right] - \E \left[ g(Z) \right]
  }{
    \varphi_C(F_n)
  }
\end{equation*}
converges to some real number. If this number is non-zero, this
implies in particular that the rate $\varphi_C(F_n)$ is optimal, in
the sense that there exist positive constants $c_1$, $c_2$ and $n_0$
such that 
\begin{equation*}
  c_1 \leq
  \frac{
    \abs{
      \E \left[ g(F_n) \right] - \E \left[ g(Z) \right]
    }
  }{
    \varphi_C(F_n)
  }
  \leq c_2
  \end{equation*}
for $n \geq n_0$. As already mentioned in the introduction, by
approximating a Lipschitz function by functions with bounded
derivatives, this implies that 
$\varphi_C(F_n)$ is optimal for the one-dimensional Wasserstein
distance (see~\cite{1196.60034} for details). By considering
coordinate projections, optimality in multiple dimensions case can
immediately be reduced to the one-dimensional case.
However, obtaining exact asymptotics is a much more
involved task, as the next two theorems show.

\begin{theorem}[Exact asymptotics for the fluctuating variance case]
  \label{thm:2}
  Assume that $\Delta_{\Gamma}(F_n) \to 0$ and let $g \colon \R^d \to
  \R$ be three times differentiable with
  bounded derivatives up to order three.
  If, for $1 \leq i,j \leq d$, the random sequences $\big( F_n,
  \widetilde{\Gamma}_{ij}(F_n) \big)_{n \geq 1}$  are ACN whenever
  $\sqrt{\on{Var} \Gamma_{ij}(F_n)} \asymp \Delta_{\Gamma}(F_n)$
  it holds that
  \begin{multline}
    \label{eq:38}
    \frac{1}{\Delta_{\Gamma}(F_n)} 
    \Big(
      \E \left[ g(F_n) \right]
      -
      \E \left[ g(Z_n) \right]
     \\    
      -
      \frac{1}{3}
      \sum_{i,j,k=1}^d
      \sqrt{\on{Var}\Gamma_{ij}(F_n)}
      \rho_{ijk,n}
      \E \left[ \partial_{ijk} g(Z_n)\right]
    \Big)
    \to 0.
  \end{multline}
  Here, the constants $\rho_{ijk,n}$ are defined by
  \begin{equation*}
    \rho_{ijk,n} = \E \left[
    \widetilde{Z}_{ij,n} Z_{k,n} \right]
  \end{equation*}
 whenever~ $\sqrt{\on{Var} 
      \Gamma_{ij}(F_n)} \asymp \Delta_{\Gamma}(F_n)$ holds and  $(F_N,
    \widetilde{\Gamma}_{ij,n})_{n \geq 1}$ is ACN 
    with corresponding Gaussian sequence
    $(Z_n,\widetilde{Z}_{ij,n})$, and $\rho_{ijk,n}=0$ otherwise. 
\end{theorem}
 
\begin{remark}
 Clearly, the condition~ $\sqrt{\on{Var} \Gamma_{ij}(F_n)} \asymp
 \Delta_{\Gamma}(F_n)$ in the above Theorem expresses the fact that we can
 neglect those    summands of $\Delta_{\Gamma}(F_n)$ that vanish ``too
 fast'' and  therefore do not contribute to the overall speed of
 convergence of $\Delta_{\Gamma}(F_n)$. 
\end{remark}

\begin{proof}[Proof of Theorem~\ref{thm:2}]
  By applying Theorem~\ref{thm:3}, we get
  \begin{multline}
    \label{eq:3}
    \frac{
      \E \left[ g(F_n) \right] - \E \left[ g(Z_n) \right]
      }{\Delta_{\Gamma}(F_n)}
    \\ =
    \sum_{i,j=1}^d
    \frac{\sqrt{\on{Var} \, \Gamma_{ij}(F_n)}}{\Delta_{\Gamma}(F_n)}
    \E  \left[
      \partial_{ij} U_{g,C_n}(F_n)
      \,
      \widetilde{\Gamma}_{ij}(F_n)
    \right].
  \end{multline}
  The bound~\eqref{eq:76} for the derivatives of $U_{g,C_n}$ and the
  fact that   $\widetilde{\Gamma}_{ij}(F_n)$ has unit variance
  immediately implies that the expectations occuring in the sum on the
  right hand side of~\eqref{eq:3} are bounded. Therefore, we only have
  to examine those summands in the same sum, for which~\eqref{eq:28}
  is true (as   all  others   vanish in the limit).
  Now choose $1 \leq i,j \leq d$ such that~\eqref{eq:28} holds. Due to
  our assumption,
  Theorem~\ref{thm:10} implies the existence of random vectors 
  $(F^{\ast}_n,\widetilde{\Gamma}_{ij}(F_n)^{\ast})$  and Gaussian
  random variables
  $(Z^{\ast}_n,\widetilde{Z}^{\ast}_{ij,n})$, defined on some common
  probability space, such that
  $(F^{\ast}_n,\widetilde{\Gamma}_{ij}(F_n)^{\ast})$ has the same law
  as $(F_n,\widetilde{\Gamma}_{ij}(F_n))$,
  $(Z^{\ast},\widetilde{Z}_{ij,n}^{\ast})$ has the same law as
  $(Z,\widetilde{Z}_{ij,n})$ and
  $(F_n^{\ast}-Z_n^{\ast},\widetilde{\Gamma}_{ij}(F_n)^{\ast} -
  \widetilde{Z}_{ij,n}^{\ast}) \to 0$ almost surely. 
  Thus we can write
  \begin{equation}
    \label{eq:5}
    \E \left[ \partial_{ij} U_{g,C_n}(F_n)
    \widetilde{\Gamma}_{ij}(F_n) \right] = \eta^1_{ij,n} +
  \eta^2_{ij,n} + \eta^3_{ij,n},
  \end{equation}
  where
  \begin{align*}
    \eta_{ij,n}^1
    &=
    \E  \left[
      \left(
        \partial_{ij} U_{g,C_n}(F_n^{\ast})
        -
        \partial_{ij} U_{g,C_n}(Z_n^{\ast})
      \right)
      \, 
      \widetilde{\Gamma}_{ij}(F_n)^{\ast}
    \right],
    \\
    \eta_{ij,n}^2
    &=
    \E  \left[
      \partial_{ij} U_{g,C_n}(Z_n^{\ast})
      \left(
        \widetilde{\Gamma}_{ij}(F_n)^{\ast}
        -
        \widetilde{Z}_{ij,n}^{\ast}
      \right)
    \right]
    \intertext{and}
    \eta_{ij,n}^3
    &=
    \E \left[
        \partial_{ij} U_{g,C_n}(Z_n^{\ast}) \widetilde{Z}_{ij,n}^{\ast}        
      \right].
    \end{align*}
  The integration by parts formula~\eqref{eq:10} and
  Lemma~\ref{lem:10}b) yield
  \begin{equation*}
    \eta_{ij,n}^3
    =
    \frac{1}{3} \sum_{k=1}^d \E \left[ Z_{k,n}^{\ast}
      \widetilde{Z}_{ij,n}^{\ast} \right] \E \left[ \partial_{ijk}
      U_{g,C_n}(Z_n^{\ast}) \right],
  \end{equation*}
  so that the proof is finished as soon as we have established that
  \begin{equation*}
      \eta_{ij,n}^1 \to 0 \quad \text{and} \quad \eta_{ij,n}^2 \to 0.
  \end{equation*}
  But this is an immediate consequence of the Lipschitz continuity of 
  $\partial_{ij} U_{g,C_n}$,  the fact that
  $\widetilde{\Gamma}_{ij}(F_n)^{\ast}$ has unit 
  variance (implying uniform integrability of 
  the sequences $(\widetilde{\Gamma}_{ij}(F_n)^{\ast})_{n \geq 0}$ and
  $(\widetilde{\Gamma}_{ij}(F_n)^{\ast} - \widetilde{Z}_{ij,n})_{n
    \geq 0}$) and the bound~\eqref{eq:76}.
\end{proof}

\begin{theorem}[Exact asymptotics for the converging variance case]
  \label{thm:5}
  Assume that $\Delta_{\Gamma}(F_n) \to 0$ and let $g \colon \R^d \to
  \R$ be three times differentiable with
  bounded derivatives up to order three.
  If there exists a covariance matrix $C$ such that
  $\Delta_C(F_n) \to 0$ and, for $1 \leq i,j \leq d$, the random
  sequences $\big( F_n, \widetilde{\Gamma}_{ij}(F_n) \big)_{n \geq 1}$
  converge in law to a centered Gaussian random vector $(Z,\widetilde{Z}_{ij})$
  whenever
  \begin{equation}
    \label{eq:19}
    \sqrt{\on{Var} \Gamma_{ij}(F_n)} + \abs{\E \left[ F_{i,n}
        F_{j,n} \right] - C_{ij}}
    \asymp
    \varphi_C(F_n)
  \end{equation}
  it holds that
  \begin{multline}
    \label{eq:42}
    \frac{1}{\varphi_C(F_n)}
    \left(
      \E \left[ g(F_n) \right]
      -
      \E \left[ g(Z) \right]
      -
      \frac{1}{2}
      \sum_{i,j=1}^d
      \left( \E \left[ F_{i,n}F_{j,n} \right] - C_{ij} \right)
      \E \left[ \partial_{ij} g(Z) \right]
    \right.
    \\ \left.
      -
      \frac{1}{3}
      \sum_{i,j,k=1}^{d}
      \sqrt{\on{Var}\Gamma_{ij}(F_n)}
      \rho_{ijk}
      \E \left[ \partial_{ijk} g(Z) \right]
    \right)
    \to 0.
  \end{multline}
  Here, the constants $\rho_{ijk}$ are defined by $\rho_{ijk} = \E
    \left[ \widetilde{Z}_{ij} Z_k \right]$ whenever~\eqref{eq:19}
    is true and $\rho_{ijk}=0$ otherwise.
\end{theorem}

\begin{proof}
  Theorem~\ref{thm:3} implies  
  \begin{multline}
    \label{eq:45}
    \frac{\E \left[ g(F_n) \right] - \E \left[ g(Z)
      \right]}{\varphi_C(F_n)}
    =
    \sum_{i,j=1}^d
    \left(
    \frac{\mu_{ij}(F_n) -C_{ij}}{\varphi_C(F_n)}
    \E \left[ \partial_{ij} U_{g,C}(F_n) \right]
    \right.
    \\ \left. +
    \frac{\sqrt{\on{Var}\Gamma_{ij}(F_n)}}{\varphi_C(F_n)}
    \E \left[ \partial_{ij} U_{g,C}(F_n) \,
      \widetilde{\Gamma}_{ij}(F_n) \right]
    \right).
  \end{multline}
  Arguing as in the proof of Theorem~\ref{thm:2}, we see
  that all expectations occuring in the sum on the right hand side
  of~\eqref{eq:45} are 
  bounded. Therefore, we can choose $1 \leq i,j \leq d$ and assume
  that~\eqref{eq:19} is true (as otherwise the corresponding summand
  would vanish   in   the limit). 
  By the boundedness of the second derivatives of $U_{g,C}$ (see~\eqref{eq:76})
  and our assumption 
  of  convergence in law, we get 
  \begin{equation*}
    \E \left[ \partial_{ij} U_{g,C}(F_n) \right] \to \E
    \left[ \partial_{ij} U_{g,C}(Z) \right]
  \end{equation*}
  and
  \begin{equation*}
    \E \left[ \partial_{ij} U_{g,C}(F_n) \widetilde{\Gamma}_{ij}(F_n)
    \right]
    \to
    \E \left[ \partial_{ij} U_{g,C}(Z) \widetilde{Z}_{ij}
    \right].
  \end{equation*}
  The integration by parts formula~\eqref{eq:10} and Lemma~\ref{lem:10}b) now yield
  \begin{equation*}
    \E \left[ \partial_{ij} U_{g,C}(Z) \right] = \frac{1}{2} \E
    \left[ \partial_{ij} g(Z) \right]
  \end{equation*}
  and
  \begin{equation*}
    \E \left[ \partial_{ij} U_{g,C}(Z) \widetilde{Z}_{ij}
    \right]
    =
    \frac{1}{3} \sum_{k=1}^d \E \left[ \widetilde{Z}_{ij} Z_k \right] \E \left[ \partial_{ijk} Z \right],
  \end{equation*}
  finishing the proof.
\end{proof}

\begin{remark}\label{s-22} 
  If the covariance $C$ of the Gaussian random variable $Z$
  is positive definite, the Hermite polynomials $H_{\alpha}(x,C)$
  form an orthonormal basis for the space $L^2(\R^d,\gamma_{C})$, where
  $\gamma_{C}$ is the density of $Z$, so that an expansion of the form
  $g(x) = \sum_{\alpha} H_{\alpha}(x,C)$ exists for all $x \in
  \R^d$. Thus, the integration by parts formula
  \begin{equation*}
    \E \left[ \partial_{\alpha} g(Z) \right] = \E
    \left[ g(Z) H_{\alpha}(Z,C) \right], 
  \end{equation*}
  valid for any multi-index $\alpha$ up to order three, yields a
  neccessary condition for the limit to be non-zero: $g$ must not be
  orthogonal (in $L^2(\R^d,\gamma_C)$) to all second- and third-order
  Hermite polynomials. 
\end{remark}

An immediate consequence of the Theorems~\ref{thm:2} and~\ref{thm:5} is
the following corollary.

\begin{corollary}[Sharp bounds and exact limits] 
  \label{cor:1} \hfill
  \begin{enumerate}[a)]
  \item   In the setting of Theorem~\ref{thm:2}, the $\liminf$ and
    $\limsup$ of the sequence
    \begin{equation*}
      \left(
    \frac{\abs{\E \left[ g(F_n) \right] - \E \left[ g(Z_n)
        \right]}}{\Delta_{\Gamma}(F_n)} \right)_{n \geq 1}    
\end{equation*}
coincide with those of the sequence
  \begin{equation*}
    \left(
    \frac{1}{3 \Delta_{\Gamma}(F_n)} \sum_{i,j,k=1}^d \sqrt{\on{Var}
      \Gamma_{ij}(F_n)} \rho_{ik,n} \E \left[ \partial_{ijk} g(Z_n)
    \right]
    \right)_{n \geq 1}.
  \end{equation*}
\item In the setting of Theorem~\ref{thm:5}, the $\liminf$ and
  $\limsup$ of the sequence
  \begin{equation}
    \label{eq:29}
    \left(
    \frac{
      \E \left[ g(F_n) \right] - \E \left[ g(Z) \right]
    }{
      \varphi_C(F_n)
    }
    \right)_{n \geq 1}
  \end{equation}
  coincide with those of
    \begin{multline}
      \label{eq:66}
      \left(
        \frac{1}{2}
        \sum_{i,j=1}^d
        \frac{
          \E \left[ F_{i,n} F_{j,n} \right] - C_{ij}
        }{
          \varphi_C(F_n)
        }
        \E \left[ \partial_{ij} g(Z) \right]
        \right.
        \\
        \left. +
        \frac{1}{3}
        \sum_{i,j,k=1}^d
        \frac{
          \sqrt{\on{Var}\Gamma_{ij}(F_n)}
        }{
          \varphi_C(F_n)
        }
        \rho_{ijk}
        \E \left[ \partial_{ijk} g(Z) \right]
      \right)_{n \geq 1}.
    \end{multline}
    In particular, if the sequence~\eqref{eq:66} converges, it
    provides the exact limit of the sequence~\eqref{eq:29}.
      \end{enumerate}
\end{corollary}

If $d=1$ and the $F_n$ all have identical variances, the assumptions
of Corollary~\ref{cor:1}b) are always satisfied and one obtains an  
    analogue of Theorem~3.1 in~\cite{1196.60034}.

If the third-order moments $\mu_{\alpha}(F_n)$ of the random vectors
$F_n$ exist, the third order Edgeworth expansion
$\mathcal{E}_3(F_n,Z,g)$, introduced in section~\ref{s-16}, is well-defined
for any Gaussian random vector $Z$ and any three-times differentiable function $g$ with
bounded  derivatives up to order three. The next theorem shows how
these expansions can be used to increase the speed of convergence.

\begin{theorem}[One-term Edgeworth expansions]
  \label{thm:1}
  Let $g \colon \R^d \to \R$ be three times differentiable with
  bounded derivatives up to order three and assume that $F_n$ has
  finite moments up to order three for $n \geq 1$, and moreover
  \begin{equation}
    \label{eq:39}
    \left(
    \frac{
      \sum_{k=1}^d \mu_{ijk}(F_n)
      }{
        \sqrt{\on{Var} \Gamma_{ij}(F_n)}}
      \right)_{n \geq 1}
  \end{equation}
  is bounded whenever
  $\sqrt{\on{Var}\Gamma_{ij}(F_n)}/\Delta_{\Gamma}(F_n) \to 0$. 
  \begin{enumerate}[a)]
  \item If all assumptions of Theorem~\ref{thm:2}a) are satisfied, it holds that
  \begin{equation}
    \label{eq:6}
    \frac{
      \E \left[ g(F_n) \right] - \mathcal{E}_3(F_n,Z_n,g)
    }{
      \Delta_{\Gamma}(F_n)
    }
    \to 0.
  \end{equation}
  \item If all assumptions of Theorem~\ref{thm:2}b) are satisfied, it
    holds that 
  \begin{equation}
    \label{eq:23}
    \frac{
      \E \left[ g(F_n) \right] - \mathcal{E}_3(F_n,Z,g)
    }{
      \varphi_C(F_n)
    }
    \to 0.
  \end{equation}
  \end{enumerate}
\end{theorem}

\begin{remark}
  The third order Edgeworth expansion $\mathcal{E}_3(F_n,Z_n,g)$
  in~\eqref{eq:6} takes the explicit form
  \begin{equation}
    \label{eq:43}
    \mathcal{E}_3(F_n,Z_n,g)
    =
    \E \left[ g(Z_n) \right]
    +
    \sum_{i,j,k=1}^d
    \frac{\mu_{ijk}(F_n)}{3!}
    \E \left[ \partial_{ijk} g(Z_n) \right]
  \end{equation}
  whereas $\mathcal{E}_3(F_n,Z,g)$ in~\eqref{eq:23} is given by
  \begin{equation}
    \label{eq:44}
    \E \left[ g(Z) \right]
    +
    \sum_{i,j=1}^d
    \frac{C_{ij} - \mu_{ij}(F_n)}{2}
    \E \left[ \partial_{ij} g(Z) \right]
    +
    \sum_{i,j,k=1}^d
    \frac{\mu_{ijk}(F_n)}{3!}
    \E \left[ \partial_{ijk} g(Z) \right]
  \end{equation}
\end{remark}

\begin{proof}[Proof of Theorem~\ref{thm:1}]
  By Theorem~\ref{thm:2}a), it is
  sufficient to show that
  \begin{equation}
    \label{eq:46}
    \frac{1}{3}
    \sum_{i,j,k=1}^d
    \frac{\sqrt{\on{Var}\Gamma_{ij}(F_n)}}{\Delta_{\Gamma}(F_n)}
    \left(
    \rho_{ijk,n}
    -
    \frac{\mu_{ijk}(F_n)}{2 \sqrt{\on{Var}\Gamma_{ij}(F_n)}}
    \right)
    \E \left[ \partial_{ijk} g(Z_n) \right]
    = 0,
  \end{equation}
  
  Fix $1 \leq i,j \leq d$. If
  \begin{equation*}
    \frac{
      \sqrt{\on{Var}\Gamma_{ij}(F_n)}
    }{
      \Delta_{\Gamma}(F_n)
    }
    \to 0,
  \end{equation*}
  the quantity
  \begin{equation*}
    \frac{\mu_{ijk}(F_n)}{2 \sqrt{\on{Var}\Gamma_{ij}(F_n)}}
  \end{equation*}
  is bounded by assumption, so that the corresponding summand in the
  sum~\eqref{eq:46} vanishes in the limit.
  If $\limsup \sqrt{\on{Var}\Gamma_{ij}(F_n)}/\Delta_{\Gamma}(F_n)$ is
  positive,
  the sequence $(F_n,\widetilde{\Gamma}_{ij,n})_{n \geq 1}$ is
  ACN. Thus, for $n \geq 1$, there exists Gaussian random variables
  $(Z_n,\widetilde{Z}_{ij,n})$ with the same covariance as
  $(F_n,\widetilde{\Gamma}_{ij,n})$.
  By definition, we get
  \begin{equation*}
    \rho_{ijk,n}
    =
    \E \left[ Z_{k,n} \widetilde{Z}_{ij,n} \right]
    =
    \frac{
          \E \left[
      F_{k,n} \Gamma_{ij}(F_n)
    \right]
      }{\sqrt{\on{Var} \Gamma_{ij}(F_n)}}
    =
    \frac{
          \E \left[
      \Gamma_{ijk}(F_n)
    \right]
      }{\sqrt{\on{Var} \Gamma_{ij}(F_n)}}.
  \end{equation*}
  The cumulant formula~\eqref{eq:60} and the fact that
  $\on{Var} \Gamma_{ij}(F_n) = \on{Var} \Gamma_{ji}(F_n)$ now yields
  \begin{equation*}
    \rho_{ijk,n} + \rho_{jik,n}
    =
    \frac{
      \kappa_{ijk}(F_n)
    }{
      \, \sqrt{\on{Var} \Gamma_{ij}(F_n)}
    }
    =
    \frac{
      \mu_{ijk}(F_n)
    }{
      \, \sqrt{\on{Var} \Gamma_{ij}(F_n)}
    }
  \end{equation*}
  so that~\eqref{eq:46} follows.

  Likewise, by Theorem~\ref{thm:5}, it is sufficient for the proof
  of assertion   b) to show that
  \begin{equation}
    \label{eq:52}
    \frac{1}{3}
    \sum_{i,j,k=1}^d
    \frac{\sqrt{\on{Var}\Gamma_{ij}(F_n)}}{\varphi_C(F_n)}
    \left(
      \rho_{ijk} - \frac{\mu_{ijk}(F_n)}{2 \sqrt{\on{Var}\Gamma_{ij}(F_n)}}
    \right)
    \E \left[ \partial_{ijk} G(Z) \right]
    \to 0,
  \end{equation}
  Again, by assumption, if $1 \leq i,j \leq d$ is
  such that $\sqrt{\on{Var}\Gamma_{ij}(F_n)}/\varphi_C(F_n) \to 0$,
  the corresponding summand in~\eqref{eq:52} vanishes in the limit.
  If, on the other hand, $\sqrt{\on{Var}\Gamma_{ij}(F_n)} \asymp
  \varphi_C(F_n)$, the sequence $(F_n,\widetilde{\Gamma}_{ij}(F_n))_{n
    \geq 1}$ converges in law to $(Z,\widetilde{Z}_{ij})$.
  Therefore, 
  \begin{equation}
    \label{eq:53}
    \E \left[ \widetilde{\Gamma}_{ij,n} F_{k,n} \right] \to
    \E \left[ \widetilde{Z}_{ij} Z_k \right]
    =
    \rho_{ijk}
  \end{equation}
  for $1 \leq k \leq d$. 
  The cumulant  formula~\eqref{eq:60} gives
  \begin{equation*}
    \frac{
      \mu_{ijk}(F_n)
    }{
      \, \sqrt{\on{Var} \Gamma_{ij}(F_n)}
    }
    =
    \frac{
      \kappa_{ijk}(F_n)
    }{
      \, \sqrt{\on{Var} \Gamma_{ij}(F_n)}
    }
    =
      \E \left[ \widetilde{\Gamma}_{ij,n} F_{k,n} \right]
      +
      \E \left[ \widetilde{\Gamma}_{ji,n} F_{k,n} \right],
  \end{equation*}
  which together with~\eqref{eq:53} implies that
  \begin{equation*}
    \frac{
      \mu_{ijk}(F_n)
    }{
      \, \sqrt{\on{Var} \Gamma_{ij}(F_n)}
    }
    \to
    \rho_{ijk} + \rho_{jik}.
  \end{equation*}
  This immediatiely yields~\eqref{eq:52}, finishing the proof.
\end{proof}

\section{The case of multiple integrals}
\label{s-6}

In this section, we specialize our results to the case where the
components of the sequence $(F_n)_{n 
  \geq 1}$ are vectors of multiple integrals. As in the previous section,
we fix an integer $d \geq 1$ and study a sequence $(F_n)_{n \geq
  1}=(F_{1,n},\ldots,F_{d,n})_{n \geq 1}$ of $\R^d$-valued random
vectors, but now each component $F_{i,n}$ is a multiple integral of
the form $F_{i,n}=I_{q_i}(f_{i,n})$ where $g_i \geq 1$ and $f_{i,n}
\in \mathfrak{H}^{\odot q_i}$. Recall the definitions for the random
variables $\Gamma_{ij}(F_n)$, $\widetilde{\Gamma}_{ij}(F_n)$ and
$Z_n$, which were given in the first paragraph of the previous
section.

Let us begin by deducing explicit representations of some of the
crucial quantities of the last section. 
Using the product formula~\eqref{prodform} and the orthogonality
   property~\eqref{orthogonality} of multiple integrals, we see that
    \begin{multline}
      \label{eq:16}
      \Gamma_{ij}(F_n) = 
      \frac{1}{q_j} \left\langle DF_{i,n},DF_{j,n}
      \right\rangle_{\mathfrak{H}}
      \\=
      \sum_{r=1}^{q_i \land q_j}
      q_i \, \beta_{q_i-1,q_j-1}(r-1) \,
      I_{q_i+q_j-2r}(f_{i,n} \widetilde{\otimes}_r f_{j,n})
    \end{multline}
    and
    \begin{multline}
           \label{eq:27}
      \on{Var} \, \Gamma_{ij}(F_n)
      \\ =
      \sum_{r=1}^{q_i \land q_j - \delta_{q_iq_j}}
      (q_i+q_j-2r)! \, q_i^2 \, \beta_{q_i-1,q_j-1}^2(r-1)
           \, \norm{f_{i,n} \widetilde{\otimes}_r
             f_{j,n}}_{\mathfrak{H}^{\otimes
               (q_i+q_j-2r)}}^2,
         \end{multline}
where the positive constants $\beta_{a,b}(r)$ are defined by~\eqref{eq:86}.

As all constants in the sum on the right hand side of~\eqref{eq:27}
are positive, this implies that
\begin{equation}
  \label{eq:81}
  \on{Var}\Gamma_{ij}(F_n) \asymp
  \sum_{r=1}^{q_i \land q_j - \delta_{q_iq_j}}
  \norm{f_{i,n}
  \widetilde{\otimes}_r f_{j,n}}^2_{\mathfrak{H}^{\otimes (q_i+q_j-2r)}}
\end{equation}
and therefore
\begin{equation}
  \label{eq:8}
  \Delta_{\Gamma}(F_n) \asymp
  \left(
  \sum_{i,j=1}^d \sum_{r=1}^{q_i \land q_j - \delta_{q_iq_j}}
  \norm{f_{i,n}
  \widetilde{\otimes}_r f_{j,n}}_{\mathfrak{H}^{\otimes
    (q_i+q_j-2r)}}^2
  \right)^{1/2}.
\end{equation}

If, for some integers $i,j,k$ with $1 \leq i,j,k \leq d$, it holds
that $r:=\frac{q_i+q_j-q_k}{2} 
\in \{ 1,2,\ldots,q_i \land q_j \}$, formula~\eqref{eq:80} and the
Cauchy-Schwarz inequality yield
\begin{align*}
  \mu_{ijk}(F_n) = \kappa_{e_i+e_j+e_k}{F_n} \leq c \norm{f_i
    \widetilde{\otimes}_r f_j}_{\mathfrak{H}^{\otimes (q_i+q_j-2r)}}
  \, \norm{f_k}_{\mathfrak{H}^{\otimes q_k}}.
\end{align*}
Combining this with~\eqref{eq:81}, we see that the
quantity~\eqref{eq:39} from Theorem~\ref{thm:1} is bounded and
therefore one-term Edgeworth expansions are always possible whenever
the corresponding ACN-conditions from Theorems~\ref{thm:2}
and~\ref{thm:5} are verified. Thus we have proven the following proposition.

\begin{proposition}[Exact asymptotics and Edgeworth expansions for multiple integrals]
  \label{prop:6}
  In the above framework, let $g \colon \R^d \to \R$ be three times
  differentiable with bounded derivatives up to order three.
  \begin{enumerate}[a)]
  \item If, for  $1 \leq i,j \leq d$, the random sequence 
    $(F_n,\widetilde{\Gamma}_{ij}(F_n))_{n \geq 1}$ is ACN
whenever
\begin{equation*}
\limsup_n \sqrt{\on{Var}\Gamma_{ij}(F_n)}/\Delta_{\Gamma}(F_n) > 0,   
\end{equation*}
  it
  holds that
  \begin{equation*}
    \frac{
      \E \left[ g(F_n) \right] - \mathcal{E}_3 (F_n,Z_n,g)
    }{
      \Delta_{\Gamma}(F_n)
    }
    \to
    0.
  \end{equation*}
\item If there exists a covariance matrix $C$ such that $\Delta_C(F_n)
  \to 0$ and, for $1 \leq i,j \leq d$, the random sequence    $(F_n,\widetilde{\Gamma}_{ij}(F_n))_{n \geq 1}$ is ACN
 whenever
 \begin{equation*}
     \limsup_n \sqrt{\on{Var}\Gamma_{ij}(F_n)}+\abs{\E \left[ F_{i,n}F_{j,n}
    \right] - C_{ij}}/\varphi_C(F_n) > 0,
 \end{equation*}
 it holds that
  \begin{equation*}
    \frac{
      \E \left[ g(F_n) \right] - \mathcal{E}_3(F_n,Z,g)
    }{
      \varphi_C(F_n)
    }
    \to 0.
  \end{equation*}
  \end{enumerate}
\end{proposition}

Sufficient conditions for the sequences
$(F_n,\widetilde{\Gamma}_{ij}(F_n))_{n \geq 1}$ to be ACN are given by
the following proposition.

\begin{proposition}
  \label{prop:7}
  The sequence $(F_n,\widetilde{\Gamma}_{ij}(F_n))_{n \geq 1}$ is ACN
if
\begin{equation}
  \label{eq:32}
  \sum_{i=1}^d \sum_{r=1}^{q_i-1}
  \, 
  \norm{f_{i,n}
  \otimes_r f_{i,n}}_{\mathfrak{H}^{\otimes 2(q_i-r)}} \to 0
\end{equation}
and
\begin{equation}
    \label{eq:26}
    \frac{
      \sum\limits_{r=1}^{q_i \land q_j - \delta_{q_iq_j}}
      \sum\limits_{s=1}^{q_i+q_j-2r-1}
      \norm{(f_{i,n} \widetilde{\otimes}_r f_{j,n}) \otimes_s
        (f_{i,n} \widetilde{\otimes}_r
        f_{j,n})}_{\mathfrak{H}^{\otimes 2 (q_i+q_j-2r-1 -
          s)}}
    }{
      \sum\limits_{r=1}^{q_i \land q_j - \delta_{q_iq_j}}
  \norm{f_{i,n}
  \widetilde{\otimes}_r f_{j,n}}^2_{\mathfrak{H}^{\otimes (q_i+q_j-2r)}}
}
\to 0.
\end{equation}
\end{proposition}

\begin{proof}
  This is a direct consequence of~\eqref{eq:16}, \eqref{eq:27} and Lemma~\ref{lem:1}.
\end{proof}

If we assume that $F_i=I_2(f_i)$ for $f_i \in \mathfrak{H}^{\odot 2}$,
$1 \leq i \leq d$, we can state  conditions for the  
ACN property of $(F_n,\widetilde{\Gamma}_{ij}(F_n))_{n \geq 1}$ which
only involve cumulants. This is due to the well known formula
(see~\cite{MR871252})
\begin{multline}
  \kappa_{k \times e_i}(F)
  = 
  2^{k-1} (k-1)! \, \on{Tr} \left( H_{f_i}^k \right)
  \\ =
  2^{k-1} (k-1)! \left\langle f_i \otimes_1^{(k-1)} f_i,f_i
  \right\rangle_{\mathfrak{H}}
  =  \label{eq:34}
  2^{k-1} (k-1)! \, \sum_{n=1}^{\infty} \lambda_{f_i,n}^k,
\end{multline}
where $H_f \colon \mathfrak{H} \to \mathfrak{H}$ is the
Hilbert-Schmidt operator defined by $H_f(g) = f \otimes_1 g$ and
$\{\lambda_{f,n} \colon n \geq 1 \}$ are its eigenvalues.

In particular, we have
\begin{multline}
  \label{eq:30}
   \kappa_{4e_i}(F)
  = 2^3 \, 3! \, \left\langle f_i \otimes_1^{(3)} f_i,f_i \right\rangle_{\mathfrak{H}}
  \\= 2^3 \, 3! \, \left\langle f_i \otimes_1 f_i, f_i \otimes_1 f_i
  \right\rangle_{\mathfrak{H}}
  = 2^3 \, 3! \, \norm{f_i \otimes_1 f_i}_{\mathfrak{H}}^2
\end{multline}
and
\begin{equation}
  \label{eq:48}
 \kappa_{8e_i}(F)
  =
  2^7 \, 7! \, \left\langle f_i \otimes_1^{(7)} f_i,f_i
  \right\rangle_{\mathfrak{H}}
  =
  2^7 \, 7! \, \norm{(f_i \otimes_1 f_i) \otimes_1 (f_i \otimes_1
    f_i)}_{\mathfrak{H}}^2.
\end{equation} 

Using Propositions~\ref{prop:6}, \ref{prop:7} and the 
Cauchy-Schwarz inequality, we can now deduce the following 
Proposition. We will however omit these elementary 
calculations, as the Proposition will also follow as a special case 
from Theorem~\ref{thm:6} of the forthcoming section. 

\begin{proposition}
  \label{prop:1}
  Let $(F_n) = (F_{1,n},\ldots,F_{d,n})$ be a sequence of random vectors whose
  components are elements of 
  the second chaos, $g \colon \R^d \to \R$ be three times differentiable with 
  bounded derivatives up to order three and assume that  
  \begin{equation}
    \label{eq:102}
    \sum_{i=1}^d \kappa_{4e_i}(F_n) \to 0.
  \end{equation}
  \begin{enumerate}[a)]
  \item If, for $1 \leq i \leq d$, it holds that $\kappa_{4e_i}(F_n)
    \asymp \sum_{i=1}^d \kappa_{4e_i}(F_n)$ implies
    \begin{equation}
      \label{eq:107}
      \frac{\kappa_{8e_i}(F_n)}{\kappa_{4e_i}(F_n)^2}
      \to 0,
    \end{equation}
    then
    \begin{equation*}
      \frac{\E \left[ g(F_n) \right] -
        \mathcal{E}_3(F_n,Z_n,g)}{\Delta_{\Gamma}(F_n)}
      \to
      0.
    \end{equation*}
  \item If there exists a covariance matrix $C$ such that
    $\Delta_C(F_n) \to 0$ and, for $1 \leq i \leq d$, the
    convergence~\eqref{eq:107} is implied by
    \begin{multline*}
     \kappa_{4e_i}(F_n) + \abs{\kappa_{e_i+e_j}(F_n)-C_{ij}} \\ \asymp
    \sum_{i=1}^d \kappa_{4e_i}(F_n) + \sqrt{\sum_{i,j=1}^d \left(
        \kappa_{e_i+e_j}(F_n) - C_{ij} \right)^2},
    \end{multline*}
    then
    \begin{equation*}
      \frac{\E \left[ g(F_n) \right] -
        \mathcal{E}_3(F_n,Z,g)}{\varphi_{C}(F_n)}
      \to
      0.
    \end{equation*}
  \end{enumerate}
\end{proposition}

Note that in the case $d=1$, part b) becomes a (weaker) version of
Proposition~3.8 from~\cite{1196.60035}. Thus, in the second chaos, the
joint speed of convergence can be compeletely characterized by the
coordinate sequences. 

\subsection{Majorizing integrals and the role of mixed contractions}
\label{s-17}

We now turn to the question whether the mixed contractions (i.e. those for
which $i \neq j$) in the numerator and denominator of
condition~\eqref{eq:26} are neccessary to ensure that 
$(F_n,\widetilde{\Gamma}_{ij}(F_n))_{n \geq 1}$ is ACN. 
It will turn out that in some cases, most notably the one where all
kernels are non-negative, we can replace condition~\eqref{eq:26}  by a
similar fraction containing only non-mixed contractions. In these 
cases, the ``interplay'' of the different kernels thus has no
influence on the speed of convergence.  

To be able to develop our theory, we assume that $\mathfrak{H} = 
L^2(A,\mathcal{A},\nu)$, where $(A,\mathcal{A})$ is a Polish space,
$\mathcal{A}$ is the associated Borel $\sigma$-field and the measure
$\nu$ is positive, $\sigma$-finite and non-atomic. This can be done
without loss of generality (see~\cite[section2.2]{MR2118863}). The
components of the random vectors $F_n$ under examination are still
multiple integrals. 

If $f_i \in \mathfrak{H}^{\odot q_i}$ and $f_j \in
\mathfrak{H}^{\odot q_j}$ are two symmetric kernels and $1 \leq r \leq
q_i \land q_j$, then according to formula~\eqref{eq:20} we can write
\begin{equation}
  \label{eq:69}
  \norm{f_i \widetilde{\otimes}_r f_j}_{\mathfrak{H}^{\otimes
      q_i+q_j-2r}}^2
  =
  \sum_{u=0}^{(q_i \land q_j) -r}
  c_u \,
  G_r(f_i,f_j,u),
\end{equation}
where the $c_u$ are some positive universal constants not depending on
$f_i$ and $f_j$ and each $G_r(f_i,f_j,u)$
is an integral of the form 
\begin{multline}\label{eq:1}
  \int_{A^u}
  \int_{A^u}
  \int_{A^r}
  \int_{A^r}
  \int_{A^n}
  \int_{A^m}  
  f_i(v,x,y)
  f_i(v,\widetilde{x},\widetilde{y})
  f_j(\widetilde{x},y,w)
  f_j(x,\widetilde{y},w)
  \\
  \diff{\mu^{\otimes m}(v)}
  \diff{\mu^{\otimes n}(w)}
  \diff{\mu^{\otimes r}(x)}
  \diff{\mu^{\otimes r}(\widetilde{x})}
  \diff{\mu^{\otimes u}(y)}
  \diff{\mu^{\otimes u}(\widetilde{y})},  
\end{multline}
where $m = q_i-r-u$ and $n = q_j-r-u$.
We can visualize each of these integrals by an integer weighted, undirected graph
\begin{equation}\label{eq:17}
  \begin{tikzpicture}[description/.style={fill=white,inner
            sep=2pt},
          baseline=(current bounding box.center)]
          \matrix (m) [matrix of math nodes, row sep=4em,
          column sep=4em, text height=1.5ex, text depth=0.25ex
          , ampersand replacement = \&]
          { f_i \& f_i \\
            f_j \& f_j \\ };
          \path[-,font=\scriptsize]
          (m-1-1) edge node[auto] {$q_i-r-u$} (m-1-2)
          (m-1-2) edge node[auto] {$r$} (m-2-2)
          (m-2-2) edge node[auto] {$q_j-r-u$} (m-2-1)
          (m-2-1) edge node[auto] {$r$} (m-1-1)
          (m-2-1) edge node[above] {$u$} (m-1-2)
          (m-1-1) edge node[below] {$u$} (m-2-2)
          ;
        \end{tikzpicture},
\end{equation}
by identifying each kernel occuring in the integral with a vertex
and drawing an edge with weight $l$ between two functions, if $l$
variables of these two functions coincide. For example, the edge with
label $q_i-r-u$ in the above graph
corresponds to the variable $v$ in the integral~\eqref{eq:1}. Due to
the symmetry of the kernels involved, we can freely translate back and
forth between the explicit notation~\eqref{eq:1} and the visual
notation~\eqref{eq:17} without losing any information. To avoid
cumbersome treatment of degenerate cases, we adopt the convention that
edges with weight zero are non-existent.

Analogously, we can write
\begin{equation}
  \label{eq:14}
  \norm{(f_i \widetilde{\otimes}_r f_j) \otimes_s (f_i
    \widetilde{\otimes}_r
    f_j)}_{\mathfrak{H}^{\otimes 2(q_i+q_j-2r-s)}}^2
  =
  \sum_{\lambda \in \Lambda}
  c_{\lambda} \, G_{r,s}(f_i,f_j,\lambda),
\end{equation}
where $\Lambda$ is some finite index set, the $c_{\lambda}$ are
positive constants and the $G_{r,s}(f_i,f_j,\lambda)$ are integrals of
the form
\begin{equation*}
  \int
  f_i(\cdot) f_i(\cdot) f_i(\cdot) f_i(\cdot)
  f_j(\cdot) f_j(\cdot) f_j(\cdot) f_j(\cdot)
\end{equation*}
involving four copies of the kernels $f_i$ and $f_j$,
respectively, which are obtained by first choosing $2(q_i+q_j)$ pairs
of variables, then identifying variables that have been paired and
finally integrating with respect to the $2(q_i+q_j)$ resulting
variables. The only constraint one has to obey is that two variables
that stem from the same kernel must not be paired.
One could write this with a lot more rigour (using, for
example, diagrams and partitions, see~\cite{MR2791919}, or a visual
method similar to ours, see~\cite{marinucci_central_2008}) but for our 
purposes it is enough to know that each of this integrals can be
visualized as a graph with eight vertices (four of them 
labeled with $f_i$  and $f_j$, respectively) that contains 
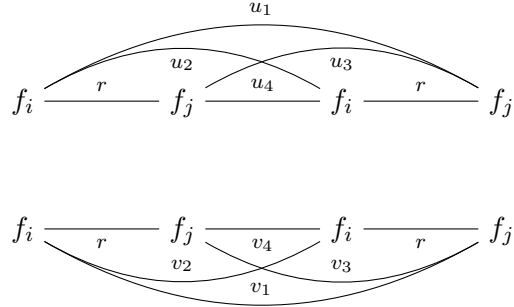
\begin{equation*}
            \begin{tikzpicture}[auto,description/.style={fill=white,inner
            sep=2pt},
          baseline=(current bounding box.center)]
          \matrix (m) [matrix of math nodes, row sep=3em,
          column sep=4em, text height=1.5ex, text depth=0.25ex
          , ampersand replacement = \&]
          { f_i \& f_j \& f_i \& f_j \\
            f_i \& f_j \& f_i \& f_j \\ };
          \path[-,font=\scriptsize]
          (m-1-1) edge node[auto] {$r$} (m-1-2)
          (m-2-2) edge node[auto] {$r$} (m-2-1)
          (m-1-3) edge node[auto] {$r$} (m-1-4)
          (m-2-4) edge node[auto] {$r$} (m-2-3)
          (m-1-1) edge [bend left] node[auto] {$u_1$} (m-1-4)
          (m-1-1) edge [bend left] node[below] {$u_2$} (m-1-3)
          (m-1-2) edge [bend left] node[below] {$u_3$} (m-1-4)
          (m-1-2) edge node[auto] {$u_4$} (m-1-3)
          (m-2-1) edge [bend right] node[auto] {$v_1$} (m-2-4)
          (m-2-1) edge [bend right] node[above] {$v_2$} (m-2-3)
          (m-2-2) edge [bend right] node[above] {$v_3$} (m-2-4)
          (m-2-2) edge node[below] {$v_4$} (m-2-3)
          ;
        \end{tikzpicture}
\end{equation*}
as a subgraph, where $\sum_i u_i = \sum_{i} v_i = s$ and some of
the $u_i$ and $v_i$ can be zero (recall our convention that an edge
with weight zero is non-existent). Note that in the original graph
there are always edges (to be precise, exactly $2(q_i+q_j-r)-s$ of
them) connecting the ``upper'' and ``lower'' groups of four
edges.

If we are given an integral $G$ occuring in the
representations~\eqref{eq:69} or~\eqref{eq:14}, we can arbitrarily
divide the four or eight kernels appearing in the integrands into two
sets $A$ and $B$ and use the Cauchy-Schwarz inequality to obtain a bound of the
type $\abs{G} \leq G_1^{1/2} G_2^{1/2}$, where the integrals $G_1$ and
$G_2$ only involve kernels in the set $A$ and $B$, respectively. Of
course, $G_1$ and $G_2$ can also be visualized by graphs. In fact,
these two graphs can be obtained without analytical detour by the  
following, purely visual ``cut-mirror-merge''-operation on the graph
of $G$:   

\begin{enumerate}[1)]
\item Divide the vertices of $G$ into two groups $A$ and $B$.
\item Erase all edges that connect vertices of different groups, thus
  obtaining two subgraphs. We refer to a vertex adjacent to an edge
  that has been erased as a \emph{bordering vertex}.
\item For each of these two subgraphs, 
  take a copy of this subgraph and connect each bordering vertex of
  the   subgraph with the corresponding vertex in the copy by an edge.
  The   weight of this edge is equal to the sum of the weights of all
  erased edges that were adjacent to this vertex and have been erased
  in step two.
\end{enumerate}

For example, starting from the integral given by the
graph~\eqref{eq:17}, if we choose two identical sets consisting of one $f_i$
and one $f_j$, respectively, the resulting graphs for $G_1$ and $G_2$
are identical as well and given by
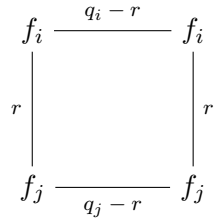
\begin{equation*}
  \begin{tikzpicture}[description/.style={fill=white,inner
            sep=2pt},
          baseline=(current bounding box.center)]
          \matrix (m) [matrix of math nodes, row sep=4em,
          column sep=4em, text height=1.5ex, text depth=0.25ex
          , ampersand replacement = \&]
          { f_i \& f_i \\
            f_j \& f_j \\ };
          \path[-,font=\scriptsize]
          (m-1-1) edge node[auto] {$q_i-r$} (m-1-2)
          (m-1-2) edge node[auto] {$r$} (m-2-2)
          (m-2-2) edge node[auto] {$q_j-r$} (m-2-1)
          (m-2-1) edge node[auto] {$r$} (m-1-1)
          ;.
        \end{tikzpicture}
\end{equation*}
Translated back into the language of integrals, this is just the  
well-known fact that
\begin{equation*}
  \norm{f_i \widetilde{\otimes}_r f_j }_{\mathfrak{H}^{\otimes (q_i+q_j-2r)}}
  \leq
  \norm{f_i \otimes_r f_j}_{\mathfrak{H}^{\otimes (q_i+q_j-2r)}},
\end{equation*}
which can of course be proven much more concisely by a direct
calculation. However, the advantage of working with graphs reveals
itself when dealing with the integrals on the right hand side
of~\eqref{eq:14}. We will see this when proving the forthcoming
Majorizing Lemma, that plays a key role in this section. 

If $f_i=f_j$, some integrals appearing in the sum on the right hand
side  of~\eqref{eq:69} are of special interest, as they dominate all 
others (in a sense that will be made clear in the sequel). We title
them \emph{majorizing integrals}. They are defined as follows.

\begin{definition}[Majorizing integrals]
  \label{def:2}
  For $f \in \mathfrak{H}^{\odot q}$, $1 \leq r \leq q-1$ and $0 \leq
  m \leq q-r$, the \emph{majorizing integrals} $M_r(f,m)$ are given by
  the graph
  \begin{equation}
    \label{eq:82}
          \begin{tikzpicture}[auto,description/.style={fill=white,inner
            sep=2pt},
          baseline=(current bounding box.center)]
          \matrix (m) [matrix of math nodes, row sep=3em,
          column sep=4em, text height=1.5ex, text depth=0.25ex
          , ampersand replacement = \&]
          { f \& f \& f \& f \\
            f \& f \& f \& f \\ };
          \path[-,font=\scriptsize]
          (m-1-1) edge node[auto] {$r$} (m-1-2)
          (m-2-2) edge node[auto] {$r$} (m-2-1)
          (m-1-3) edge node[auto] {$r$} (m-1-4)
          (m-2-4) edge node[auto] {$r$} (m-2-3)
          (m-1-1) edge [bend left] node[above] {$m$} (m-1-3)
          (m-1-2) edge [bend left] node[above] {$m$} (m-1-4)
          (m-2-1) edge [bend right] node[below] {$m$} (m-2-3)
          (m-2-2) edge [bend right] node[below] {$m$} (m-2-4)
          (m-1-1) edge node[auto] {$q-r-m$} (m-2-1)
          (m-1-2) edge node[auto] {$q-r-m$} (m-2-2)
          (m-1-3) edge node[auto] {$q-r-m$} (m-2-3)
          (m-1-4) edge node[auto] {$q-r-m$} (m-2-4)
          ;
        \end{tikzpicture}.
      \end{equation}
\end{definition}

Observe that the integrals $M_r(f,m)$ are non-negative and appear in
the expansion of the type~\eqref{eq:14} for the norm $\norm{(f \widetilde{\otimes}_r
  f) \otimes_s (f \widetilde{\otimes}_r f)}_{\mathfrak{H}^{\otimes 4(q-r)-2s}}$.
Also, by  grouping the inner four vertices in the graph~\eqref{eq:82} and  
applying ``cut-mirror-merge'', we see that
\begin{equation*}
  M_r(f,m) \leq \norm{f \otimes_r f}_{\mathfrak{H}^{\otimes 2(q-r)}}^4,
\end{equation*}
with equality if $m \in \{0,q-r\}$. 

\begin{lemma}[Majorizing Lemma]
  \label{lem:7}
  Let $q_i$ and $q_j$ be two positive integers and $f_i \in
  \mathfrak{H}^{\odot q_i}$, $f_j \in \mathfrak{H}^{\odot q_j}$ be two
  symmetric kernels. For given integers $r$ and $s$ with $1 \leq r
  \leq q_i \land q_j - \delta_{q_i,q_j}$ and $1 \leq s \leq q_i + q_j
  - 2r-1$, let $G_{r,s}(f_i,f_j,\lambda)$ be one of the summands in the
  representation~\eqref{eq:14}. Then, for $1 \leq k \leq 4$, there
  exists integers $n_k \in \{0,1,\ldots,q_i-r\}$ and $m_k \in \{
  0,1,\ldots,q_j-r\}$ with $m_k+n_k=s$, such that
  \begin{equation}
  \label{eq:11}
  G_{r,s}(f_i,f_j,\lambda)^8
  \leq
  \prod_{k=1}^4
  \big(
    M_r(f_i,m_k)
    \,
    M_r(f_j,n_k)
  \big).
\end{equation}
\end{lemma}

\begin{proof}
  We will iteratively apply Cauchy-Schwarz (using the visual method
  developed above) to obtain the chain 
  \begin{equation}
    \label{eq:93}
    G_{r,s}(f_i,f_j,\lambda)^8
    \leq
    \left( G_1' \, G_1'' \right)^4
    \leq
    \left( G_{2,1}' \, G_{2,2}' \, G_{2,1}'' \, G_{2,2}'' \right)^2
    \leq
    \prod_{k=1}^4
  \big(
    M_r(f_i,m_k)
    \,
    M_r(f_j,n_k)
  \big),
\end{equation}
  where all primed and double-primed quantities are integrals which
  will be described by their corresponding graphs. 
  
As already mentioned, a graph associated with an integral
$G_{r,s}(f_i,f_j,\lambda)$ in the sum~\eqref{eq:14} has eight vertices
(four of them labeled with $f_i$, the other four with $f_j$) and contains
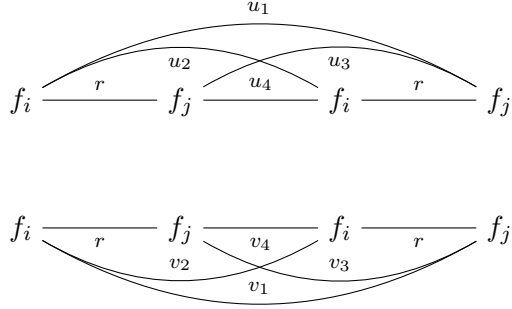
\begin{equation*}
            \begin{tikzpicture}[auto,description/.style={fill=white,inner
            sep=2pt},
          baseline=(current bounding box.center)]
          \matrix (m) [matrix of math nodes, row sep=3em,
          column sep=4em, text height=1.5ex, text depth=0.25ex
          , ampersand replacement = \&]
          { f_i \& f_j \& f_i \& f_j \\
            f_i \& f_j \& f_i \& f_j \\ };
          \path[-,font=\scriptsize]
          (m-1-1) edge node[auto] {$r$} (m-1-2)
          (m-2-2) edge node[auto] {$r$} (m-2-1)
          (m-1-3) edge node[auto] {$r$} (m-1-4)
          (m-2-4) edge node[auto] {$r$} (m-2-3)
          (m-1-1) edge [bend left] node[auto] {$u_1$} (m-1-4)
          (m-1-1) edge [bend left] node[below] {$u_2$} (m-1-3)
          (m-1-2) edge [bend left] node[below] {$u_3$} (m-1-4)
          (m-1-2) edge node[auto] {$u_4$} (m-1-3)
          (m-2-1) edge [bend right] node[auto] {$v_1$} (m-2-4)
          (m-2-1) edge [bend right] node[above] {$v_2$} (m-2-3)
          (m-2-2) edge [bend right] node[above] {$v_3$} (m-2-4)
          (m-2-2) edge node[below] {$v_4$} (m-2-3)
          ;
        \end{tikzpicture}
\end{equation*}
as a subgraph, where $\sum_i u_i = \sum_{i} v_i = s$ and some of
the $u_i$ and $v_i$ can be zero (recall our convention that an edge with weight
zero is non-existent).
We now apply the Cauchy-Schwarz inequality for the first time,
grouping the four vertices connected by the $u_i$- and
$v_i$-edges. The resulting bounding integrals $G_1'$ and $G_1''$ are
given by the graph 
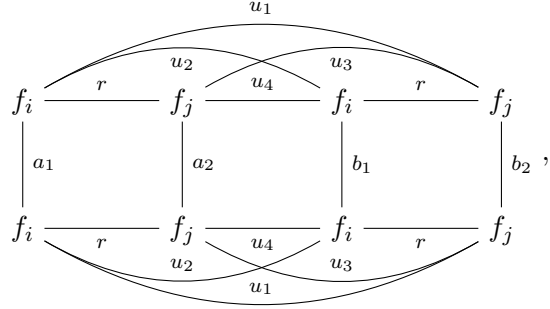
\begin{equation*}
            \begin{tikzpicture}[auto,description/.style={fill=white,inner
            sep=2pt},
          baseline=(current bounding box.center)]
          \matrix (m) [matrix of math nodes, row sep=3em,
          column sep=4em, text height=1.5ex, text depth=0.25ex
          , ampersand replacement = \&]
          { f_i \& f_j \& f_i \& f_j \\
            f_i \& f_j \& f_i \& f_j \\ };
          \path[-,font=\scriptsize]
          (m-1-1) edge node[auto] {$r$} (m-1-2)
          (m-2-2) edge node[auto] {$r$} (m-2-1)
          (m-1-3) edge node[auto] {$r$} (m-1-4)
          (m-2-4) edge node[auto] {$r$} (m-2-3)
          (m-1-1) edge [bend left] node[auto] {$u_1$} (m-1-4)
          (m-1-1) edge [bend left] node[below] {$u_2$} (m-1-3)
          (m-1-2) edge [bend left] node[below] {$u_3$} (m-1-4)
          (m-1-2) edge node[auto] {$u_4$} (m-1-3)
          (m-2-1) edge [bend right] node[auto] {$u_1$} (m-2-4)
          (m-2-1) edge [bend right] node[above] {$u_2$} (m-2-3)
          (m-2-2) edge [bend right] node[above] {$u_3$} (m-2-4)
          (m-2-2) edge node[below] {$u_4$} (m-2-3)
          (m-1-1) edge node[auto] {$a_1$} (m-2-1)
          (m-1-2) edge node[auto] {$a_2$} (m-2-2)
          (m-1-3) edge node[auto] {$b_1$} (m-2-3)
          (m-1-4) edge node[auto] {$b_2$} (m-2-4)

          ;
        \end{tikzpicture},
\end{equation*}
where $a_1+a_2+b_1+b_2=2(q_i+q_j-r)-s>0$       
and the same graph with $u_i$ replaced by $v_i$. We now continue to
apply Cauchy-Schwarz to $G_1'$. The exact same operations then have to
be performed with $G_1''$ to obtain the final result.

In the graph of $G_1'$, we group the four vertices connected by $a_i$- and
$b_i$-edges respectively and then apply Cauchy-Schwarz. This yields
bounding integrals  $G_{2,1}'$, given by a ``cube'' of the form
\begin{equation}
  \label{eq:92}
            \begin{tikzpicture}[auto,description/.style={fill=white,inner
            sep=2pt},
          baseline=(current bounding box.center)]
          \matrix (m) [matrix of math nodes, row sep=3em,
          column sep=4em, text height=1.5ex, text depth=0.25ex
          , ampersand replacement = \&]
          { f_i \& f_j \& f_i \& f_j \\
            f_i \& f_j \& f_i \& f_j \\ };
          \path[-,font=\scriptsize]
          (m-1-1) edge node[auto] {$r$} (m-1-2)
          (m-2-2) edge node[auto] {$r$} (m-2-1)
          (m-1-3) edge node[auto] {$r$} (m-1-4)
          (m-2-4) edge node[auto] {$r$} (m-2-3)
          (m-1-1) edge [bend left] node[above] {$\widetilde{u}_1$} (m-1-3)
          (m-1-2) edge [bend left] node[above] {$\widetilde{u}_2$} (m-1-4)
          (m-2-1) edge [bend right] node[below] {$\widetilde{u}_1$} (m-2-3)
          (m-2-2) edge [bend right] node[below] {$\widetilde{u}_2$} (m-2-4)
          (m-1-1) edge node[auto] {$a_1$} (m-2-1)
          (m-1-2) edge node[auto] {$a_2$} (m-2-2)
          (m-1-3) edge node[auto] {$a_1$} (m-2-3)
          (m-1-4) edge node[auto] {$a_2$} (m-2-4)
          ;
        \end{tikzpicture},
\end{equation}
where $0 \leq a_1+a_2 \leq 2(q_i+q_j-r)-s$,
$0 \leq \widetilde{u}_1,\widetilde{u_2}\leq s$ and
$\widetilde{u}_1+\widetilde{u}_2=s$, and $G_{2,2}'$, given by the same
graph with the $a_i$ replaced by $b_i$. 

From the graphs for $G'_{2,1}$ and $G_{2,2}'$, by grouping the four
$f_i$- and $f_j$-vertices together and then applying Cauchy-Schwarz
another time, we now obtain graphs that represent majorizing
integrals. For example, starting from
the graph~\eqref{eq:92} for $G_{2,1}'$, we obtain  
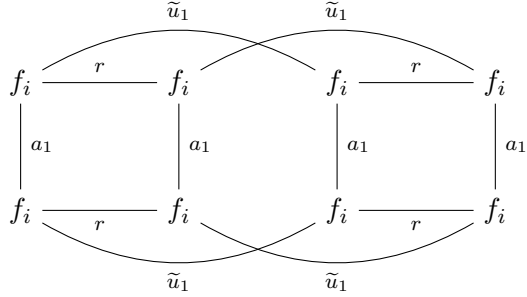
\begin{equation*}
          \begin{tikzpicture}[auto,description/.style={fill=white,inner
            sep=2pt},
          baseline=(current bounding box.center)]
          \matrix (m) [matrix of math nodes, row sep=3em,
          column sep=4em, text height=1.5ex, text depth=0.25ex
          , ampersand replacement = \&]
          { f_i \& f_i \& f_i \& f_i \\
            f_i \& f_i \& f_i \& f_i \\ };
          \path[-,font=\scriptsize]
          (m-1-1) edge node[auto] {$r$} (m-1-2)
          (m-2-2) edge node[auto] {$r$} (m-2-1)
          (m-1-3) edge node[auto] {$r$} (m-1-4)
          (m-2-4) edge node[auto] {$r$} (m-2-3)
          (m-1-1) edge [bend left] node[above] {$\widetilde{u}_1$} (m-1-3)
          (m-1-2) edge [bend left] node[above] {$\widetilde{u}_1$} (m-1-4)
          (m-2-1) edge [bend right] node[below] {$\widetilde{u}_1$} (m-2-3)
          (m-2-2) edge [bend right] node[below] {$\widetilde{u}_1$} (m-2-4)
          (m-1-1) edge node[auto] {$a_1$} (m-2-1)
          (m-1-2) edge node[auto] {$a_1$} (m-2-2)
          (m-1-3) edge node[auto] {$a_1$} (m-2-3)
          (m-1-4) edge node[auto] {$a_1$} (m-2-4)
          ;
        \end{tikzpicture}
      \end{equation*}
with $0 \leq \widetilde{u}_1 \leq s$ and
$a_1=q_i-r-\widetilde{u}_1$ and the same graph with $f_i$,
$\widetilde{u}_1$ and $a_1$ replaced by $f_j$, $\widetilde{u}_2$ and
$a_2$, respectively.

Starting from $G''_1$, the integrals  $G''_{2,1}$ and
$G''_{2,1}$ as well as the corresponding majorizing integrals are
obtained analogously. Finally, a careful inspection of the single
steps indicated above yields that the majorizing integrals obtained
after the final application of the Cauchy-Schwarz inequality are indeed
of the form stated in~\eqref{eq:11}.
\end{proof}

We are now ready to prove the main theorem of this section.

\begin{theorem}
  \label{thm:6}
  Let $g \colon \R^d \to \R$ be three times differentiable with
  bounded derivatives up to order three and assume that the following
  conditions are true.
  \begin{enumerate}[(i)]
  \item
      \begin{equation*}
    \sum_{i=1}^d \sum_{r=1}^{q_i-1} \norm{f_{i,n} \otimes_r
      f_{i,n}}_{\mathfrak{H}^{\otimes 2(q_i-r)}}
    \to 0.
  \end{equation*}
\item
  For those $i,j \in \{ 1,2,\ldots,d\}$for which
  $\sqrt{\on{Var}\Gamma_{ij}(F_n)}  \asymp \Delta_{\Gamma}(F_n)$
  it holds that
    \begin{equation}
    \label{eq:24}
    \sum_{r=1}^{q_i-1} \norm{f_{i,n} \widetilde{\otimes}_r
      f_{i,n}}_{\mathfrak{H}^{\otimes 2(q_i-r)}}
    \asymp
    \sum_{r=1}^{q_i-1} \norm{f_{i,n} \otimes_r
      f_{i,n}}_{\mathfrak{H}^{\otimes 2(q_i-r)}}.
  \end{equation}
  and
  \begin{equation}
    \label{eq:85}
    \frac{
      \sum_{r=1}^{q_i-1}
      \sum_{s=1}^{q_i-r-1}
      M_r (f_{i,n},s)
    }{
      \sum_{r=1}^{q_i-1}
      \norm{f_{i,n} \otimes_r f_{i,n}}_{\mathfrak{H}^{\otimes 2(q_i-r)}}^4
    }
    \to 0.
  \end{equation}
    \end{enumerate}
  Then it holds that
  \begin{equation}
    \label{eq:88}
    \frac{
      \E \left[ g(F_n) \right] - \mathcal{E}_3 (F_n,Z_n,g)
    }{
      \Delta_{\Gamma}(F_n)
    }
    \to
    0.
  \end{equation}
  If, in addition, there exists a covariance matrix $C$ such that
  $\Delta_C(F_n) \preccurlyeq \Delta_{\Gamma}(F_n) $, then
  \begin{equation}
    \label{eq:89}
    \frac{
      \E \left[ g(F_n) \right] - \mathcal{E}_3(F_n,Z,g)
    }{
      \varphi_C(F_n)
    }
    \to 0.
  \end{equation}
\end{theorem}

\begin{proof}
  Let $i,j \in \{ 1,2,\ldots, d \}$ such that 
  $\sqrt{\on{Var}\Gamma_{ij}(F_n)}  \asymp \Delta_{\Gamma}(F_n)$ and
  assume, without loss of generality, that $q_i \leq q_j$.
    
  We will show that 
  \begin{equation}
    \label{eq:90}
    \on{Var}\Gamma_{ij}(F_n) \succcurlyeq  
    \sqrt{\on{Var}\Gamma_{ii}(F_n)\on{Var}\Gamma_{jj}(F_n)}
  \end{equation}
  and 
  \begin{multline}
    \label{eq:91}
    \sum_{r=1}^{q_i \land q_j -\delta_{q_iq_j}}
    \sum_{s=1}^{q_i+q_j-2r-1}
    \norm{(f_{i,n} \widetilde{\otimes}_r f_{j,n}) \otimes_s (f_{i,n}
    \widetilde{\otimes}_r f_{j,n})}_{\mathfrak{H}^{\otimes
      2(q_i+q_j-2r-s)}}^2
  \\ \preccurlyeq 
    \sqrt{A_{1,n}} + \sqrt{A_{2,n}} + \sqrt{A_{3,n}} \,,
\end{multline}
where
\begin{align*}
  A_{1,n} &=
  \left(
  \sum_{r=1}^{q_i-1} \sum_{s=1}^{2(q_i-r)-1}
  M_r(f_{i,n},s)
\right)
  \left(
  \sum_{r=1}^{q_j-1} \sum_{s=1}^{2(q_j-r)-1}
  M_r(f_{j,n},s)
  \right)
  \\ A_{2,n} &=
  \left(
    \sum_{r=1}^{q_i-1}
    \norm{f_{i,n} \otimes_r f_{i,n}}_{\mathfrak{H}^{\otimes 2(q_i-r)}}^4
  \right)
  \left(
  \sum_{r=1}^{q_j-1} \sum_{s=1}^{2(q_j-r)-1}
  M_r(f_{j,n},s)
\right)
\intertext{and}
  A_{3,n} &=
  \left(
    \sum_{r=1}^{q_j-1}
    \norm{f_{j,n} \otimes_r f_{j,n}}_{\mathfrak{H}^{\otimes 2(q_j-r)}}^4
  \right)
  \left(
  \sum_{r=1}^{q_i-1} \sum_{s=1}^{2(q_i-r)-1}
  M_r(f_{i,n},s)
\right)
\end{align*}

   As a consequence, we get
   \begin{multline*}
      \frac{
      \sum_{r=1}^{q_i \land q_j - \delta_{q_iq_j}}
      \sum_{s=1}^{q_i+q_j-2r-1}
      \norm{(f_{i,n} \widetilde{\otimes}_r f_{j,n}) \otimes_s
        (f_{i,n} \widetilde{\otimes}_r
        f_{j,n})}_{\mathfrak{H}^{\otimes 2 (q_i+q_j-2r-1 -
          s)}}
    }{
      \sum_{r=1}^{q_i \land q_j - \delta_{q_iq_j}}
  \norm{f_{i,n}
  \widetilde{\otimes}_r f_{j,n}}^2_{\mathfrak{H}^{\otimes (q_i+q_j-2r)}}
}
 \\ \preccurlyeq
  \left(
  \frac{
    \sqrt{A_{1,n}} + \sqrt{A_{2,n}} + \sqrt{A_{3,n}}
  }{
    \on{Var} \Gamma_{ii}(F_n) \on{Var} \Gamma_{jj}(F_n)
    }
  \right)^{1/2}
  \to 0,
\end{multline*}
where the convergence to zero is implied by assumptions~\eqref{eq:85}
and~\eqref{eq:24}. In view of Proposition~\ref{prop:7}, this proves~\eqref{eq:88}. By
the same argument we obtain~\eqref{eq:89}, as $\Delta_C(F_n)
\preccurlyeq \Delta_{\Gamma}(F_n)$ implies that
\begin{equation*}
  \sqrt{\on{Var}\Gamma_{ij}(F_n)} + \abs{\E \left[ F_{i,n} F_{j,n}
  \right] - C_{ij}} \asymp \varphi_C(F_n)
\end{equation*}
 is equivalent to
$\sqrt{\on{Var}\Gamma_{ij}(F_n)}  \asymp
\Delta_{\Gamma}(F_n)$. 

To prove~\eqref{eq:90}, note that the Cauchy-Schwarz inequality
implies for $1 \leq r < q_i$ that
\begin{align*}
  \norm{f_{i,n} \widetilde{\otimes}_r f_{j,n}}_{\mathfrak{H}^{\otimes
      (q_i+q_j-2r)}}^2
  &\leq
  \norm{f_{i,n} \otimes_r f_{j,n}}_{\mathfrak{H}^{\otimes
      (q_i+q_j-2r)}}^2
  \\ &\leq
  \norm{f_{i,n} \otimes_r f_{i,n}}_{\mathfrak{H}^{\otimes
      2(q_i-r)}}
  \,
  \norm{f_{j,n} \otimes_r f_{j,n}}_{\mathfrak{H}^{\otimes
      2(q_j-r)}}.
\end{align*}
Therefore,
\begin{align*}
  \sum_{r=1}^{q_i \land q_j - 1}
  &\norm{f_{i,n} \widetilde{\otimes}_r f_{j,n}}_{\mathfrak{H}^{\otimes
      (q_i+q_j-2r)}}^2
  \\ &\qquad \leq
  \sum_{r=1}^{q_i \land q_j - 1}
  \norm{f_{i,n} \otimes_r f_{i,n}}_{\mathfrak{H}^{\otimes
      2(q_i-r)}}
  \,
  \norm{f_{j,n} \otimes_r f_{j,n}}_{\mathfrak{H}^{\otimes
      2(q_j-r)}} 
  \\ &\qquad\leq
  \left( \sum_{r=1}^{q_i-1} \norm{f_{i,n} \otimes_r
      f_{i,n}}_{\mathfrak{H}^{\otimes 2(q_i-r)}} \right)
  \left( \sum_{r=1}^{q_j-1} \norm{f_{j,n} \otimes_r
      f_{j,n}}_{\mathfrak{H}^{\otimes 2(q_j-r)}} \right)
\end{align*}
Together with assumption~\eqref{eq:24}, this gives
\begin{align*}
  \on{Var}_{ij}(F_n) 
  &=
  \sum_{r=1}^{q_i \land q_j - \delta_{q_iq_j}}
  \norm{f_{i,n} \widetilde{\otimes}_r
    f_{j,n}}^2_{\mathfrak{H}^{\otimes (q_i+q_j-2r)}} 
\\ &\preccurlyeq
  (1-\delta_{q_iq_j}) \norm{f_{i,n} \widetilde{\otimes}_{q_i}
    f_{j,n}}^2_{\mathfrak{H}^{\otimes (q_j-q_i)}}
  +
  \sqrt{
    \on{Var}_{ii}(F_n)
    \on{Var}_{jj}(F_n)
  }
  \\ &\preccurlyeq
    (1-\delta_{q_iq_j}) \norm{f_{i,n} \widetilde{\otimes}_{q_i}
    f_{j,n}}^2_{\mathfrak{H}^{\otimes (q_j-q_i)}}
  +
  \on{Var}_{ii}(F_n)
  +
  \on{Var}_{jj}(F_n)  
  \\ &\preccurlyeq
  \Delta_{\Gamma}(F_n)^2,
\end{align*}
where we set $\mathfrak{H}^{\otimes 0} = \R$.
As $\sqrt{\on{Var}_{ij}(F_n)} \asymp \Delta_{\Gamma}(F_n)$, this shows
\begin{equation*}
  \on{Var}_{ij}(F_n) 
  \asymp
  (1-\delta_{q_iq_j}) \norm{f_{i,n} \widetilde{\otimes}_{q_i}
    f_{j,n}}^2_{\mathfrak{H}^{\otimes (q_j-q_i)}}
  +
  \sqrt{
    \on{Var}_{ii}(F_n)
    \on{Var}_{jj}(F_n)
  }
\end{equation*}
and therefore~\eqref{eq:90}.

Now let $1 \leq r \leq q_i \land q_j - \delta_{q_iq_j}$, $1 \leq s \leq
q_i+q_j-2r-1$ and $G_{r,s}(f_{i,n},f_{j,n},\lambda)$ be an integral from the 
right hand side of representation~\eqref{eq:14}. By the Majorizing
Lemma~\ref{lem:7}, we can find integers $n_k \in \{ 0,1,\ldots,q_i-r \}$ and $m_k \in \{
0,1,\ldots,q_j-r\}$ such that $m_k + n_k=s$ for $1
\leq k \leq 4$ and
\begin{equation*}
  G_{r,s}(f_{i,n},f_{j,n},\lambda) \leq
  \prod_{k=1}^4
  \big(
  M_r(f_{i,n},m_k) M_r(f_{j,n},n_k)
  \big)^{1/8}.
\end{equation*}
As clearly
\begin{equation*}
  M_r(f_{i,n},m_k) \, M_r(f_{j,n},n_k)
  \leq
  A_{1,n} + A_{2,n} +A_{3,n},
\end{equation*}
this gives
\begin{equation*}
  G_{r,s}(f_{i,n},f_{j,n},\lambda) 
  \leq
  \prod_{k=1}^4
  \left( A_{1,n} + A_{2,n} + A_{3,n} \right)^{1/8}
  \leq
  \sqrt{A_{1,n}} + \sqrt{A_{2,n}} + \sqrt{A_{3,n}}.
\end{equation*}
The representation~\eqref{eq:14} thus yields
\begin{equation*}
      \norm{(f_{i,n} \widetilde{\otimes}_r f_{j,n}) \otimes_s
        (f_{i,n} \widetilde{\otimes}_r
        f_{j,n})}_{\mathfrak{H}^{\otimes 2 (q_i+q_j-2r -
          s)}}^2
      \preccurlyeq
        \sqrt{A_{1,n}} + \sqrt{A_{2,n}} + \sqrt{A_{3,n}} \, ,
\end{equation*}
wich immediately implies~\eqref{eq:91}.
\end{proof}

\begin{remark}
  \begin{enumerate}[a)]
  \item  Note that~\eqref{eq:24} is satisfied, if the kernels $f_{i,n}$ are
  either all non-negative or all non-positive for $n \geq
  n_0$.
\item In the case where $q_i=2$ for $1 \leq i \leq d$, we obtain
  Proposition~\ref{prop:1} as a special case of Theorem~\ref{thm:6}.
  Indeed,~\eqref{eq:24} is trivially satisfied, and, by~\eqref{eq:30}
  and~\eqref{eq:48}, conditions (i) and (ii) are equivalent
  to~\eqref{eq:102} and~\eqref{eq:107}, respectively.
  \end{enumerate}
\end{remark}

\section{Examples}
\label{s-24}

In this section, we provide several examples that illustrate our
techniques. 

\subsection{Step functions and matrix representations}
\label{s-1}

We start with a counterexample, that in a way shows that the kernels
involved can not be too ``simple'' in order for our techniques to work. 
Let $\mathfrak{H}=L^2([0,1),\mu)$, where $\mu$ is the Lebesgue measure
and partition $[0,1)$  into
$N$ equidistant intervals $\alpha_1,\alpha_2,\ldots,\alpha_N$ where
$
  \alpha_k =
  \left[
    \frac{k-1}{N}, \,
    \frac{k}{N}
  \right)
  $
for $k=1,\ldots,N$. Using this partition, we endow
$[0,1)^2$ with a grid and define a symmetric kernel $f \in  
\mathfrak{H}^{\odot 2}$ that is constant on each sector by
\begin{equation}
  \label{eq:70}
  f(x,y)
  =
  \sum_{i,j=1}^N
  a_{ij} \, 1_{\alpha_i}(x) 1_{\alpha_j}(y), 
\end{equation}
where the $a_{ij}$ are real constants and $a_{ij}=a_{ji}$.  Of course, $f$ is uniquely determined by the symmetric matrix
$A = (a_{ij})_{1 \leq i,j \leq N}$.
If $g$ is another kernel of the type~\eqref{eq:70}, given by a
matrix $B = (b_{ij})_{1 \leq i,j \leq N}$, we have
\begin{align*}
  (f \otimes_1 g) (x,y)
  &=
  \int_0^1
  \left(
  \sum_{i,j=1}^N
  a_{ij} \, 1_{\alpha_i}(t) 1_{\alpha_j}(x)
\right)
\left(
  \sum_{k,l=1}^N
  b_{kl} \, 1_{\alpha_k}(t) 1_{\alpha_l}(y)
  \right)
  \diff{\mu(t)}
  \\ &=
  \sum_{i,j,l=1}^N
  a_{ij} \, b_{jl}
  \,
  \mu(\alpha_j) \,
  1_{\alpha_i}(x) 1_{\alpha_l}(y)
  \\ &=
  \frac{1}{N}
  \sum_{i,j,l=1}^N
  a_{ij} \, b_{jl}
  1_{\alpha_i}(x) 1_{\alpha_l}(y)
\end{align*}
and
\begin{equation*}
  (f \widetilde{\otimes}_1 g)(x,y)
  =
  \frac{1}{2N}
  \sum_{i,j,l=1}^N
  (a_{ij} \, b_{jl} + a_{lj} \, b_{ji}) \,
  1_{\alpha_i}(x) 1_{\alpha_l}(y).
\end{equation*}
Therefore, $f \otimes_1 g$ can be identified with the matrix $C=\frac{1}{N} AB$ 
and $f \widetilde{\otimes}_1 g$ by $\frac{1}{2} \left( C + C^T
  \right)$, where $C^T$ denotes the transpose of $C$. Analogously, one
  can show that 
  \begin{equation*}
    \left\langle f,g \right\rangle_{\mathfrak{H}^{\otimes
        2}}
    =
    \frac{1}{N^2}
    \left\langle A,B \right\rangle_{H.S.}
    =
    \frac{
      \on{tr}(AB^T)
    }{N^2}.    
\end{equation*}

By formula~\eqref{eq:34}, it is now easy to see that for any $m \geq 2$ we get
\begin{equation}
  \label{eq:41}
  \kappa_m(I_2(f)) = 2^{m-1} (m-1)! \frac{\on{tr}A^m}{N^m}
\end{equation}
and therefore, by the Cauchy-Schwarz inequality,
\begin{equation}
  \label{eq:2}
  \frac{ 3!^2
  }{2 \times 7!} \,
  \frac{\kappa_8(I_2(f))}{\kappa_4(I_2(f))^2}
  =
 \frac{
   \on{tr} \left( A^8 \right)
}{
  \on{tr} \left( A^4 \right)^2
}
=
\left(
  1 +
  \frac{
    \sum_{1\leq k,l \leq N, k \neq l} \lambda_k^4 \lambda_l^4
  }{
    \sum_{1 \leq k \leq N} \lambda_k^8
  }
\right)^{-1}
\geq
\frac{1}{2},
\end{equation}
where $(\lambda_k)_{1 \leq k \leq N}$ denotes the eigenvalue sequence of $A$. 
Now fix $d \geq 1$ and choose a sequence $(N_n)$ of positive integers
greater than two. For $n \geq 1$, define random vectors
$F_n=(I_2(f_{1,n}),\ldots,I_2(f_{d,n})$, where the kernels $f_{i,n}$
are given by symmetric $(N_n \times N_n)$-matrices $A_{i,n}$, $1 \leq
i \leq d$, such that 
\begin{equation*}
  \frac{\on{tr} A_{i,n}^4}{N_n^4} \to 0 \qquad (n \to \infty)  
\end{equation*}
for $1 \leq i \leq d$. Then, by the Fourth Moment Theorem~\ref{thm:7}, $F_n$
converges in distribution to a $d$-dimensional, 
centered Gaussian random vector. However, by~\eqref{eq:2},
condition~\eqref{eq:107} is never satisfied and thus
Proposition~\ref{thm:6} fails to provide optimal rates of convergence. 

\begin{remark}
  \label{s-foll-expl-example}
  The following explicit example illustrates how symmetrization can drastically
  increase the speed of convergence (see~\cite{nourdin_asymptotic_2011},
  Remark 3.2(3) for another example): If $N=2$ and the 
kernels  $f$ and $g$ are represented by the matrices
\begin{equation*}
 A=
\begin{bmatrix}
  1 &\phantom{-}0 \\ 0 &-1
\end{bmatrix}
\qquad \text{and} \qquad
B=
\begin{bmatrix}
  0 &1 \\ 1&0
\end{bmatrix},
\end{equation*}
respectively, it holds that $AB + BA=0$ and $\on{tr}(ABBA)=2$. Therefore, $\norm{f \widetilde{\otimes}_1 g}_{\mathfrak{H}^{\otimes
    2}}^2 = 0$ while $\norm{f \otimes_1 g}_{\mathfrak{H}^{\otimes
    2}}^2 = \frac{1}{8}$.
\end{remark}
\subsection{Exploding integrals of Brownian sheets}
\label{s-11}

  Let $W= \{ W(t_1,\ldots,t_l) \colon 0 \leq t_1,\ldots,t_l \leq 1 \}$
  be a standard Brownian sheet on $[0,1]^l$, i.e. a centered Gaussian
  process such that
  \begin{equation*}
    \E \left[ W(s_1,\ldots,s_l) W(t_1,\ldots,t_l) \right]
    =
    \prod_{i=1}^l (s_i \land t_i)
  \end{equation*}
  for all $(s_1,\ldots,s_l),(t_1,\ldots,t_l) \in [0,1]^l$. We can
  identify the Gaussian space generated by $W$ with an isonormal
  process $X$ on $\mathfrak{H} = L^2([0,1]^l)$ via
  \begin{equation*}
       W(t_1,\ldots,t_l) \sim X \left( \prod_{i=1}^l \, 1_{[0,t_i]}
       \right).
  \end{equation*}

  For positive $\varepsilon$, we now define 
  \begin{equation}
    \label{eq:33}
    F_{\varepsilon} = \int_{[0,1]^l}
    \frac{
    W(t_1,\ldots,t_l)^2
  }{
  \left( t_1 \, t_2 \cdots t_l \right)^{2-\varepsilon}
}
\, 
\diff{\, (t_1,\ldots,t_l)}.
\end{equation}
  An application of
  Jeulin's lemma (see~\cite[Lemma 1, p. 44]{MR604176}) shows that $F_{\varepsilon}$
  ``explodes'' in the  limit, i.e. that $P$-almost surely  
  \begin{equation*}
    F_{\varepsilon} \to \infty \qquad (\varepsilon \to 0).
  \end{equation*}
  However, for the normalized sequence $\widetilde{F}_{\varepsilon}$,
  defined by 
  \begin{equation}
    \label{eq:4}
    \widetilde{F}_{\varepsilon} = \frac{F_{\varepsilon} - \E \left[ F_{\varepsilon} \right]}{\sqrt{\E \left[ F_{\varepsilon}^2 \right]}},
  \end{equation}
  the central limit theorem
  \begin{equation}
    \label{eq:36}
    \widetilde{F}_{\varepsilon} \xrightarrow{\mathcal{L}} Z \sim
    \mathcal{N}(0,1) \qquad (\varepsilon \to 0).
  \end{equation}
  holds. This is a consequence of the Fourth Moment
  Theorem~\ref{thm:7} and the forthcoming formula~\eqref{eq:95} that
  provides asymptotics for the cumulants of
  $\widetilde{F}_{\varepsilon}$. For slightly different exploding
  functionals of the above type, an analogous central limit theorem
  was established in~\cite{peccati_hardys_2004} for the case
  $l=1$,~\cite{deheuvels_quadratic_2006} for the case $l=2$
  and~\cite{MR2118863} for the case $l>2$. Exact asymptotics in the
  Kolmogorov distance were provided in~\cite{1196.60034}. Here, we are
  interested in vectors of such functionals. 
 
  Routine calculations show that
  \begin{equation*}
    \widetilde{F}_{\varepsilon} = \frac{F_{\varepsilon} -
      \mu_{\varepsilon}}{\sigma_{\varepsilon}} = 
    \frac{1}{\sigma_{\varepsilon}} I_2(f_{\varepsilon}),
  \end{equation*}
  where
  \begin{equation}
    \label{eq:68}
    f_{\varepsilon}(x_1,\ldots,x_l,y_1,\ldots,y_l)
    =
    \prod_{k=1}^l
    \int_0^1 \frac{1_{[0,t_k]}(x_k)
      1_{[0,t_k]}(y_k)}{t_k^{2-\varepsilon}} \diff{t_k},
  \end{equation}
  $\mu_{\varepsilon}=\E \left[ \widetilde{F}_{\varepsilon} \right]$,
  and $\sigma_{\varepsilon}^2 = \E \left[
    \widetilde{F}_{\varepsilon}^2 \right]$.

  From~\eqref{eq:68} we conclude that if
  $\varepsilon_1,\varepsilon_2,\ldots,\varepsilon_k$ are $k$ positive
  numbers, then
  \begin{equation}
    \label{eq:100}
    \left\langle
      ( \cdots ((f_{\varepsilon_1} \widetilde{\otimes}_1 f_{\varepsilon_2})
      \widetilde{\otimes}_1 f_{\varepsilon_3}) \widetilde{\otimes}_1
      \cdots )\widetilde{\otimes}_1 f_{\varepsilon_{k-1}}
      , f_{\varepsilon_k}
    \right\rangle_{\mathfrak{H}}
    =
    C(\varepsilon_1,\varepsilon_2,\ldots,\varepsilon_k)^l,
  \end{equation}
  where
  \begin{multline}
    \label{eq:57}
    C(\varepsilon_1,\varepsilon_2,\ldots,\varepsilon_k)
    \\= \int_{[0,1]^k}
    \frac{
      (s_1 \land s_2) (s_2 \land s_3) \cdots (s_{k-1} \land s_k) (s_k
      \land s_1) 
    }{s_1^{2-\varepsilon_1} s_2^{2-\varepsilon_2} \cdots s_k^{2-\varepsilon_k}} \diff{\,(s_1,\ldots,s_k)}.
  \end{multline}
  For convenience, we will
    write $C_k(\varepsilon) = C(\varepsilon,\ldots,\varepsilon)$, if
    all $k$ arguments are equal. With this notation, we have
    $\mu_{\varepsilon} = C(\varepsilon)^l$ and
    $\sigma^2_{\varepsilon} = C_2(\varepsilon)^l$.
  By partitioning the $k$-dimensional unit interval into
  simplexes, the integral~\eqref{eq:57} can be computed explicitly. These
  calculations yield that
  \begin{equation}
    \label{eq:101}
      C(\varepsilon_1,\ldots,\varepsilon_k)
      =
      \widetilde{c}(\varepsilon_1,\ldots,\varepsilon_k)
      \times
      \frac{
        k!
      }{
        \varepsilon_1+\ldots+\varepsilon_k
      }.
    \end{equation}
    Here, $\widetilde{c}(\varepsilon_1,\ldots,\varepsilon_k)$ is
    the canonical symmetrization of
    \begin{equation*}
      c(\varepsilon_1,\ldots,\varepsilon_k)
      =
      \big(
        (1+\varepsilon_1)
        (1+\varepsilon_1+\varepsilon_2)
        \cdots
        (1+\varepsilon_1+\varepsilon_2+\ldots + \varepsilon_{k-1})
        \big)^{-1}.
    \end{equation*}
    Observe that $0 < \widetilde{c}(\varepsilon_1,\ldots,\varepsilon_k) < 1$ and
    $\widetilde{c}(\varepsilon_1,\ldots,\varepsilon_k) \to 1$ if
    $\varepsilon_1,\ldots,\varepsilon_k \to 0$.

    Now fix $d \geq 1$ and, for
    $\varepsilon_1,\ldots,\varepsilon_d>0$, define
    $\widetilde{F}_{(\varepsilon_1,\ldots,\varepsilon_d)} =
    (\widetilde{F}_{\varepsilon_1},\ldots,\ldots,\widetilde{F}_{\varepsilon_d})$.
    As the inner     product on the left hand side of~\eqref{eq:100}
    does not depend on the order in which the $k$ kernels
    $f_{\varepsilon_1},f_{\varepsilon_2},\ldots,f_{\varepsilon_k}$ are
    contracted, the cumulant formula~\eqref{eq:37} yields that
  \begin{equation}
    \label{eq:94}
    \kappa_{\alpha}(\widetilde{F}_{(\varepsilon_1,\ldots,\varepsilon_d)})
    =
    2^{\abs{\alpha}-1} \, (\abs{\alpha}-1)! \, 
    C(\varepsilon_{l_1},\varepsilon_{l_2},\ldots,\varepsilon_{l_{\abs{\alpha}}})^l
    \prod_{k=1}^{\abs{\alpha}} C_2(\varepsilon_{l_k})^{-l/2},
  \end{equation}
  where $\alpha \in \N_0^d$ with elementary   
  decomposition $\{ l_1,\ldots,l_{\abs{\alpha}}
  \}$. Identity~\eqref{eq:101} shows that 
  \begin{equation}
    \label{eq:95}
    \kappa_{\alpha}(\widetilde{F}_n)^{-l}
    \asymp
      \frac{
        \left(\varepsilon_{l_1} \cdots \varepsilon_{l_{\abs{\alpha}}}
        \right)^{1/2}
      }{
        \varepsilon_{l_1} + \cdots + \varepsilon_{l_{\abs{\alpha}}}
      }.
  \end{equation}
  This allows us to apply Proposition~\ref{thm:6} to the
  vector $\widetilde{F}_{(\varepsilon_1,\ldots,\varepsilon_d)}$. We
  obtain the following explicit result.

  \begin{proposition}
    \label{prop:3}
    For $\varepsilon_1,\ldots,\varepsilon_d > 0$, let
    $\widetilde{F}_{\varepsilon_1,\ldots,\varepsilon_d} = 
    (\widetilde{F}_{\varepsilon_1},\ldots,\widetilde{F}_{\varepsilon_d})$ be
    defined by~\eqref{eq:4} and let $g
    \colon \R^2 \to \R$ be a three times 
    differentiable function with bounded derivatives up to order three.
    Then it holds that
    \begin{equation}
      \label{eq:96}
      \frac{
        \E \left[
          g(\widetilde{F}_{\varepsilon_1,\ldots,\varepsilon_d})
        \right] -
        \mathcal{E}_3(\widetilde{F}_{\varepsilon_1,\ldots,\varepsilon_d},Z_{\varepsilon_1,\ldots,\varepsilon_d}
        ,g)   
      }{
        \sqrt{
          \sum_{i=1}^d \varepsilon_i^{\, l}
          }
      }
      \to 0, \qquad \left( \varepsilon_1,\ldots,\varepsilon_d \to 0 \right),
    \end{equation}
    where  $Z_{\varepsilon_1,\ldots,\varepsilon_d}$ is a centered Gaussian random
    variable with the same covariance matrix as
    $\widetilde{F}_{(\varepsilon_1,\ldots,\varepsilon_d)}$.
    If, in addition, for $1 \leq i,j \leq d$ it holds that
    \begin{equation}
      \label{eq:40}
        \frac{
          2}{
        \sqrt{\frac{\varepsilon_i}{\varepsilon_j}}
        +
        \sqrt{\frac{\varepsilon_j}{\varepsilon_i}}
       }
      \to C_{ij}^{1/l}, \qquad \left(
        \varepsilon_1,\ldots,\varepsilon_d \to 0 \right), 
    \end{equation}
    and
    \begin{equation}
      \label{eq:97}
      \frac{
        \sum_{i,j=1}^d
        \left( 
          \left(
            \sqrt{\frac{\varepsilon_i}{\varepsilon_j}}
          +
          \sqrt{\frac{\varepsilon_j}{\varepsilon_i}}
        \right)^{-l}
        - \frac{C_{ij}}{2^l}
         \right)^2
       }{
         \sum_{i=1}^d \varepsilon_i^{\, l}
       }
       \to 0,
       \qquad \left(
        \varepsilon_1,\ldots,\varepsilon_d \to 0 \right), 
    \end{equation}
    where $C_{ij} \geq 0$, then
    \begin{equation*}
      \frac{\E \left[
          g(\widetilde{F}_{(\varepsilon_1,\ldots,\varepsilon_d)})
        \right] - \mathcal{E}_3(\widetilde{F}_{(\varepsilon_1,\ldots,\varepsilon_d)},Z,g)
      }{
        \sqrt{
          \sum_{i=1}^d \varepsilon_i^{\, l}
        }
      }
      \to 0 \qquad (\varepsilon_1,\ldots,\varepsilon_d \to 0)
    \end{equation*}
    and
    \begin{multline*}
      \label{eq:98}
      \lim_{\varepsilon_1,\ldots,\varepsilon_d \to 0}
      \frac{\E \left[
          g(\widetilde{F}_{(\varepsilon_1,\ldots,\varepsilon_d)})
        \right] - \E \left[ g(Z) \right] 
        }{
        \varphi_C(\widetilde{F}_{(\varepsilon_1,\ldots,\varepsilon_d)})
      }
      \\ =
      4 \sqrt{6}
      \lim_{\varepsilon_1,\ldots,\varepsilon_d \to 0}
      \frac{
        \sum_{i,j,k=1}^d
        \left(
          \sqrt{\frac{\varepsilon_i}{\varepsilon_j \varepsilon_k}}
          +
          \sqrt{\frac{\varepsilon_j}{\varepsilon_i \varepsilon_k}}
          +
          \sqrt{\frac{\varepsilon_k}{\varepsilon_i \varepsilon_j}}
        \right)^{-l}
        \E \left[ \partial_{ijk}g(Z) \right]
      }{
        \sqrt{
        \sum_{i,j=1}^d
        \left( \frac{1}{\varepsilon_i} + \frac{1}{\varepsilon_j}
        \right)^{-l}
        }
      }.
    \end{multline*}
    Here,
        $\varphi_C(\widetilde{F}_{(\varepsilon_1,\ldots,\varepsilon_d)})^2
    \asymp \sum_{i=1}^d\varepsilon_i^{\, l}$
    and
    $Z$ is a $d$-dimensional, centered Gaussian random variable with
    covariance $C=(C_{ij})_{i,j=1}^d$.
  \end{proposition}

  Note that by~\eqref{eq:94} and~\eqref{eq:101}, the Edgeworth
  expansions 
  $\mathcal{E}_3(\widetilde{F}_{\varepsilon_1,\ldots,\varepsilon_d}
  ,Z_{\varepsilon_1,\ldots,\varepsilon_d} ,g)$ and
  $\mathcal{E}_3(\widetilde{F}_{\varepsilon_1,\ldots,\varepsilon_d} 
  ,Z,g)$ can be calculated explicitly.

  To illustrate this, we choose a positive sequence
  $(a_n)_{n \geq 1}$ converging to zero and two positive numbers $\xi$ and
  $\zeta$. Then it holds that $\widetilde{F}_{(\xi \cdot a_n, \zeta \cdot a_n)}$
  has covariance
  \begin{equation*}
    C =
    \begin{pmatrix}
      1 & \rho^l \\
      \rho^l &1 
    \end{pmatrix}
  \end{equation*}
  for all $n \geq 1$, where
  \begin{equation*}
    \rho = \frac{2}{\sqrt{\frac{\xi}{\zeta}} + \sqrt{\frac{\zeta}{\xi}}}.
  \end{equation*}
  Thus, the conditions~\eqref{eq:40} and~\eqref{eq:97} are trivially satisfied
  and all conclusions of Proposition~\ref{prop:3} hold.   

\subsection{Continuous time Toeplitz quadratic functionals}
\label{s-23}

Let $(X_t)_{t \geq 0}$ be a centered, real valued Gaussian process
with a covariance function $r$ of the form
$ r(t) = \E \left[ X_u X_{u+t} \right] = \widehat{f}(t)$,
where $f \colon \R \to \R$ is an integrable, even function,
customarily called the spectral density of the process $(X_t)$ and
$\widehat{f}$ denotes its
Fourier transform $t \mapsto \int_{-\infty}^{\infty} \me^{\mi xt} f(x)
\diff{x}$.
If $h \colon \R \to \R$ is another integrable even function with Fourier
transform $\widehat{h}$ and $T > 0$, we define the Toeplitz functional
$Q_{h,T}$ associated with $h$ and~$T$ by
\begin{equation*}
  Q_{h,T} := \int_{[0,T]^2} \widehat{h}(t-s) X_t X_s \diff{(s,t)}.
\end{equation*}
and denote a normalized version by
\begin{equation*}
  \widetilde{Q}_T =
  \frac{Q_{h,T} - \E \left[ Q_{h,T} \right]}{\sqrt{T}}.
\end{equation*}
In the following, we want to apply our results to sequences of random
vectors whose components are (normalized) Toeplitz functionals,
analogous to the treatment in~\cite{1196.60034} for the
one-dimensional case.

For $T >0$ and $\psi \in L^1(\R)$, the \emph{truncated
  Toeplitz operator} $B_T(\psi)$, defined on $L^2(\R)$, is given by
\begin{equation*}
  B_T(\psi)(u)(t) = \int_0^T u(x) \widehat{\psi}(t-x) \diff{x}.
\end{equation*}
As usual, if $\psi_1,\ldots,\psi_m \in L^1(R)$ and $j \geq 1$, we write
\begin{equation*}
  B_T(\psi_m) B_T(\psi_{m-1}) \cdots B_T(\psi_1)
  =
  B_T(\psi_m) \circ B_T(\psi_{m-1}) \circ \cdots  \circ B_T(\psi_1)
\end{equation*}
and 
\begin{equation*}
  \left(
  B_T(\psi_2) B_T(\psi_1) \right)^j
  =
  \underbrace{B_T(\psi_2) B_T(\psi_1) \ldots B_T(\psi_2) B_T(\psi_1)}_{j \text{ times}}
\end{equation*}
for its operator products and powers, respectively. Explicitly, the
above operator product takes the form
\begin{multline*}
  \big( B_T(\psi_m) B_T(\psi_{m-1}) \cdots B_T(\psi_1)(u) \big)(t)
  \\ =
  \int_0^T \cdots \int_0^T \int_0^T
  u(x_1)
  \widehat{\psi}_1(x_2-x_1)
  \widehat{\psi}_2(x_3-x_2)
  \cdots
  \widehat{\psi}_m(t-x_m)
  \diff{x_1}
  \diff{x_2}
  \cdots
  \diff{x_m}.
\end{multline*}

To adapt the setting to our framework, we introduce the Hilbert space
 of complex-valued, square 
integrable and even functions $h \colon \R \to \C$, equipped with the
inner product $\left\langle h_1,h_2 \right\rangle_{\mathfrak{H}} =
\int_{-\infty}^{\infty} h_1(x) \overline{h_2(x)} f(x) \diff{x}$. The
process $(X_t)_{t \geq 0}$ can then be identified with an isonormal
Gaussian process on the real subspace $\mathfrak{H}$ generated by the family
$\{ x \mapsto \me^{\mi xt} \colon t > 0 \}$. This allows us to
represent the normalized Toeplitz functional $\widetilde{Q}_{h,T}$ as
a multiple integral of second order with kernel
$\varphi_{h,T}$, which is given by 
\begin{align}
  \label{eq:74}
  \varphi_{h,T}(x_,y)
  &=
  \int_{[0,T]^2} \widehat{h}(t-s) 
  \me^{\mi (sx+ty)}
  \diff{s} \diff{t}.
\end{align}

We now choose even functions $h_1,\ldots,h_d \in L^1(\R)$ be even
functions and define a random vector $F_T=(F_{1,T},\ldots,F_{d,T})$ by
setting   $F_{i,T}=\widetilde{Q}_{h_i,T}$ for $1 \leq i \leq d$ and $T
> 0$. The following Theorem collects some results from the literature.

\begin{theorem}
  \label{thm:8}
  Let $\alpha \in \N_0^d$ be a multi-index with $\abs{\alpha}  
    \geq 2$ and elementary decomposition $\{ \,
    l_1,\ldots,l_{\abs{\alpha}} \, \}$. Then the following is true.
  \begin{enumerate}[a)]
  \item The cumulant $\kappa_{\alpha}(F_T)$ is given by
    \begin{equation}
      \label{eq:78}
    \kappa_{\alpha}(F_T)
    =
    T^{-\abs{\alpha}/2} \,
    2^{\abs{\alpha}-1} \, (\abs{\alpha}-1)! \,
    \on{Tr} \left[
      B_T(f)^{\abs{\alpha}} \prod_{i=1}^{\abs{\alpha}} B_T(h_{l_i)})
      \right].
  \end{equation}
\item If $f \in L^1(\R) \cap L^{q_0}(\R)$ and $h_{i} \in L^1(\R) \cap
  L^{q_i}(\R)$ such that $1/q_0 + 1/q_i \leq 1/\abs{\alpha}$ for $1
  \leq i \leq d$, then
  \begin{equation}
    \label{eq:79}
    \lim_{T \to \infty}
    T^{\abs{\alpha}/2-1 }\kappa_{\alpha}(F_T)
    =
    2^{\abs{\alpha}-1} (\abs{\alpha}-1)!
    \int_{-\infty}^{\infty} f^{\abs{\alpha}}(x) \prod_{i=1}^{\abs{\alpha}}
    h_{l_i}(x) \diff{x}.
  \end{equation}
\item If $f \in L^1(\R) \cap L^{q_0}(\R)$ and $h_{i} \in L^1(\R) \cap
  L^{q_i}(\R)$ such that $1/q_0 + 1/q_i \leq 1/2$ for $1
  \leq i \leq d$, then
  \begin{equation}
    \label{eq:105}
    F_T \xrightarrow{d} \mathcal{N}(0,C), \qquad (T \to \infty),
  \end{equation}
  where the covariance matrix $C=(C_{ij})_{1 \leq i,j \leq d}$ is
  given by
  \begin{equation*}
    C_{ij}= 2 \int_0^{\infty} f^2(x) h_i(x) h_j(x) \diff{x}.
  \end{equation*}
  \end{enumerate}
\end{theorem}

\begin{proof}
  Part a) follows from a straightforward adaptation of the arguments
  in~\cite[Chapter 11]{MR890515} to multiple dimensions. Part b)
  follows from a) and~\cite[Theorem
  1a)]{ginovian_toeplitz_1994}. Finally, part~c) is a consequence of
  part b) and the Fourth Moment Theorem~\ref{thm:7}. For the one-dimensional case ($d=1$), part c) was first proven  
in~\cite{ginovian_toeplitz_1994}. Weaker conditions for the
convergence~\eqref{eq:105} to take place can be found in~\cite{MR2299719}. 
\end{proof}

We are now able to prove the following Edgeworth expansion.

\begin{proposition}
  \label{prop:5}
  In the above framework, assume that $f \in L^1(\R) \cap
  L^{q_0}(\R)$, $h_{i} \in L^1(\R) \cap   L^{q_i}(\R)$ such that 
  $1/q_0 + 1/q_i \leq 1/8$ for $1 \leq i \leq d$ and let $g \colon
  \R^d \to \R$ be three times differentiable with bounded derivatives
  up to order three. Then it holds that
  \begin{equation}
    \label{eq:35}
    \frac{
      \E \left[ g(F_T) \right]
      -
      \mathcal{E}_3(F_T,Z,g)
    }{
      \varphi_C(F_T)
    }
    \to 0
    \quad (T \to \infty),
  \end{equation}
  where $Z$ is a $d$-dimensional Gaussian random variable with
  covariance matrix $C=(C_{ij})_{1 \leq i,j \leq d}$ given by
  $C_{ij}= 2 \int_0^{\infty} f^2(x) h_i(x) h_j(x) \diff{x}$.
\end{proposition}

\begin{proof}
  By Theorem~\ref{thm:8}b), we immediately verify the conditions of
  Proposition~\ref{thm:6}b) and obtain the result.
\end{proof}

\subsection{Edgeworth expansions for the Breuer-Major Theorem}
\label{s-8}

Define $B = \{B_x \colon x \geq 0 \}$ to be a fractional Brownian motion with
Hurst index $H \in 
(0,\frac{1}{2})$, i.e. a centered Gaussian process with covariance
\begin{equation*}
  \E \left[ B_xB_y \right]
  =
  \frac{1}{2} \left(
    x^{2H} + y^{2H} - \abs{x-y}^{2H}
  \right),
  \qquad
  x,y \geq 0.
\end{equation*}
For fixed $H \in (0,\frac{1}{2})$, the Gaussian space generated by $B$
can be identified with an isonormal Gaussian process $X = \{ X(h)
\colon h \in \mathfrak{H} \}$, where the real and separable Hilbert
space $\mathfrak{H}$ is the closure of the set of all $\R$-valued step
functions on $\R_{+}$ with respect to the inner product
\begin{equation*}
  \left\langle 1_{[0,x]},1_{[0,y]} \right\rangle_{\mathfrak{H}}
  :=
  \E \left[ B_xB_y \right].
\end{equation*}
In particular, we have $B_x = X(1_{[0,x]})$ for $x > 0$. For more
details on fractional Brownian motion see for
example~\cite{pecnour_mall}, \cite{MR2200233} or~\cite{nourdin_fbm_2012}. We
denote  
the covariance function of the stationary increment process $(B_{x+1} - B_x)_{x
  \geq 0}$ by
\begin{equation*}
  \rho(t) := \E \left[ B_{x+1-x}
    \left( B_{x+1+t}-B_{x+t} \right)
  \right]
  =
  \frac{1}{2} \left(
    \abs{t+1}^{2H} + \abs{t-1}^{2H} - 2 \abs{t}^{2H}
  \right).
\end{equation*}
Now choose $d$ integers $q_i \geq 2$ and, for $T>0$, define the
centered random vectors $F_T=(F_{1,T},\ldots,F_{d,T})$ by
\begin{equation}
  \label{eq:47}
  F_{i,T} =
  \frac{1}{\sqrt{T}}
  \int_0^T
  H_{q_i}(B_{u+1}-B_u)
  \diff{u}.
\end{equation}
We immediately see that the covariance matrix $C_T = (C_{ij,T})_{1 \leq
i,j \leq d}$ of $F_T$ is given by
\begin{equation*}
  C_{ij,T} =
  \E \left[
    F_{i,T} F_{j,T}
  \right]
  =
  \delta_{q_iq_j}
  \frac{q_i!}{T}
  \int_{[0,T]^2}
  \rho^{q_i}(u-v) \diff{(u,v)}
\end{equation*}
and converges to $C = (C_{ij})_{1 \leq i,j \leq d}$ for $T \to
\infty$, where
\begin{equation*}
  C_{ij} = \delta_{q_iq_j} \, q_i!
  \int_{-\infty}^{\infty} \rho^{q_i}(x) \diff{x}.
\end{equation*}

It is well known (see for example~\cite{MR716933} or~\cite{MR799146})
that for each component $F_{i,T}$ the central limit theorem
\begin{equation*}
  F_{i,T} \xrightarrow{\mathcal{L}} \mathcal{N}(0,C_{ii}),
  \qquad (T \to \infty)
\end{equation*}
holds and the Fourth Moment Theorem~\ref{thm:7}A therefore implies the
joint convergence of $F_T$ towards a centered $d$-dimensional Gaussian
random vector $Z$ with covariance $C$.
By applying our methods, we are able to derive a fluctuating and
non-fluctuating Edgeworth expansion for $F_T$, which in many cases yield exact asymptotics. 

\begin{theorem}
  \label{thm:9}
  If, in the above setting, $g$ is three times differentiable with
  bounded derivatives up to order three it holds that
  \begin{equation}
    \label{eq:73}
    \frac{
      \E \left[ g(F_T) \right]
      -
      \mathcal{E}_3(F_T,Z,g)
    }{
      1/\sqrt{T}
    }
    \to 0, \qquad (T \to \infty).
  \end{equation}

\end{theorem}

\begin{remark}
  For a non-trivial application of 
  Theorem~\ref{thm:9} at least one of the integers $q_i$ should be
  even. Indeed, otherwise the   Edgeworth expansion
  $\mathcal{E}_3(F_T,Z_T,g)$ (or $\mathcal{E}_3(F_T,Z,g)$,
  respectively) would merely reduce to the expectation $\E \left[
    g(Z_T) \right]$ (or $\E \left[ g(Z) \right]$). 
\end{remark}

\begin{proof}[Proof of Theorem~\ref{thm:9}].
We want to apply Theorem~\ref{thm:6}.
Observe that by the well-known relation $H_{q_i}(B_{u+1}-B_u) =
I_q(1^{\otimes q_i}_{[u,u+1]})$ we can represent each component
$F_{i,T}$ by a multiple integral $I_{q_i}(f_{i,T})$, where the kernels
$f_{i,T}$ are given by   
\begin{equation*}
  f_{i,T} = \frac{1}{\sqrt{T}}
  \int_0^T
  1^{\otimes q}_{[u,u+1]}
  \diff{u}.  
\end{equation*}
Straightforward calculations now yield
\begin{align*}
  &\norm{f_{i,T} \otimes_r f_{i,T}}_{\mathfrak{H}^{\otimes 2(q_i-r)}}^2
  \\ &\qquad = 
  \frac{1}{T^2}
  \int_{[0,T]^4}
  \rho(v_1-u_1)^r \rho(v_2-u_2)^r
    \\ &\qquad \qquad\qquad\qquad\qquad
  \rho(v_2-v_1)^{q_i-r} \rho(u_2-u_1)^{q_i-r}
  \diff(u_1,u_2,v_1,v_2)
  \\ &\qquad \asymp
  \frac{1}{T}
  \int_{[-T,T]^3}
  \rho(x_1)^r \rho(x_2)^r \rho(x_3)^{q_i-r}
  \rho(x_1-x_2+x_3)^{q_i-r}
  \diff(x_1,x_2,x_3),
\end{align*}
where the integral in the last line, which has been obtained by a
change of variables, converges for $T \to \infty$ (due to the
restriction $H \in (0,1/2)$). As symmetrizing only changes the
exponents of the factors of the integrand, we see that 
\begin{equation}
  \label{eq:50}
  \norm{f_{i,T} \otimes_r f_{i,T}}_{\mathfrak{H}^{\otimes 2(q_i-r)}}
  \asymp
  \norm{f_{i,T} \widetilde{\otimes}_r f_{i,T}}_{\mathfrak{H}^{\otimes
      2(q_i-r)}}
  \asymp
  \frac{1}{\sqrt{T}}.
\end{equation}
Moreover, for $1 \leq r \leq q_i-1$ and $1 \leq m \leq q_i-r-1$, we
see that the majorizing integrals $M_r(f_{i,T},s)$ are given
by \begin{align*} 
  M_r(f_{i,T},s)
  &=
  \frac{1}{T^4}
  \int_{[0,T]^8}^{}
  \big(
    \rho(v_1-u_1) \rho(v_2-u_2) \rho(v_3-u_3) \rho(v_4-u_4)
    \big)^r
  \\ &\qquad \qquad \qquad
  \big(
    \rho(u_2-u_1) \rho(v_2-v_1) \rho(u_4-u_3) \rho(v_4-v_3)
    \big)^s
    \\ &\qquad \qquad \qquad
  \big(
    \rho(u_3-u_1) \rho(v_3-v_1) \rho(u_4-u_2)
    \rho(v_4-v_2)
    \big)^{q_i-r-s}
    \\ &\qquad \qquad \qquad
    \diff{(u_1,\ldots,u_4,v_1,\ldots,v_4)}
    \\ &\asymp
  \frac{1}{T^3}
  \int_{[-T,T]^7}^{}
  \big(
    \rho(x_1) \rho(x_2) \rho(x_3) \rho(x_4)
    \big)^r
  \\ &\qquad \qquad \qquad
  \big(
    \rho(x_5) \rho(x_2+x_5-x_1) \rho(x_6) \rho(x_4+x_6-x_3)
    \big)^s
    \\ &\qquad \qquad \qquad
  \big(
  \rho(x_7) \rho(x_3+x_7-x_1) \rho(x_6+x_7-x_5)
  \\ &\qquad \qquad \qquad \qquad \qquad \qquad
    \rho(x_4+x_6+x_7-x_5-x_2)
    \big)^{q_i-r-s}
    \\ &\qquad \qquad \qquad
    \diff{(x_1,\ldots,x_7)}.
\end{align*}
By the same argument as before, the integral converges for $T \to
\infty$ and we get
\begin{equation}
  \label{eq:51}
  M_r(f_{i,T},s) \asymp \frac{1}{T^3}.
\end{equation}
Together with the asymptotic relation~\eqref{eq:50}, this shows that
conditions conditions (i) and (ii) of Theorem~\ref{thm:6} are
satisfied. It remains to show that
  $\Delta_C(F_T) \preccurlyeq \Delta_{\Gamma}(F_T)$, so that~\eqref{eq:73} follows.
  by Theorem~\ref{thm:6}.  A linear change of
  variables allows us to write
  \begin{equation*}
    C_{ij,T} - C_{ij} =
    \delta_{q_iq_j}
    \frac{q_i!}{T}
    \int_0^T
    \int_{(-\infty,-v) \cup (T-v,\infty)}^{}
    \rho^{q_i}(w) \diff{w} \diff{v}.
  \end{equation*}
  Using the well known asymptotic relation $\rho(t) \asymp
  t^{2(H-1)}$, we get 
  \begin{equation*}
    C_{ij,T} - C_{ij} \asymp \delta_{q_iq_j}
  T^{2(H-1)q_i+1}
  \end{equation*}
  and thus
 \begin{equation*}
  \Delta_C(F_T) \asymp T^{1+2(H-1) q_{\text{min}}}.
\end{equation*}
 As, by~\eqref{eq:50}, $\Delta_{\Gamma}(F_T) \asymp 1/\sqrt{T}$ and
 $1+2(H-1)q_{\text{min}} < -1/2$, the proof is finished.
\end{proof}

\section*{Acknowledgements}
\label{s-19}
The author heartily thanks Giovanni Peccati for many fruitful discussions and
remarks, Denis Serre for providing the elegant example in
Remark~\ref{s-foll-expl-example} and an anonymous referee for the other elegant 
example hinted at in the same remark and for pointing out several mistakes in a
previous version. 

\bibliographystyle{amsalpha}
\bibliography{refs}

\end{document}